\documentclass[a4paper,12pt,leqno]{amsart}
\usepackage{latexsym}
\usepackage[all]{xy}

\usepackage{amssymb} 
\usepackage{amsmath} 
\usepackage{amsthm}
\usepackage{amscd}
\usepackage{color}
\usepackage{comment}
\usepackage{mathtools}
\usepackage{mathrsfs}
\usepackage{enumerate}
\usepackage{graphicx}
\usepackage{stmaryrd}
\usepackage{comment}
\usepackage{bm}
\usepackage{url}
\usepackage{enumitem}

\usepackage{hyperref} 
\hypersetup{colorlinks=true}

\definecolor{gray}{gray}{0.7}
\definecolor{Gray}{gray}{0.3}

\usepackage{amscd}

\numberwithin{equation}{section}

\theoremstyle{break}
 \newtheorem{theorem}{Theorem}[section]
 \newtheorem{proposition}[theorem]{Proposition}
 \newtheorem{corollary}[theorem]{Corollary}
 \newtheorem{lemma}[theorem]{Lemma}
 
 \newtheorem{conjecture}[theorem]{Conjecture}
 
 \newtheorem{question}[theorem]{Question}

 \theoremstyle{definition}
 \newtheorem{definition}[theorem]{Definition}
 \newtheorem{remark}[theorem]{Remark}
 \newtheorem{example}[theorem]{Example}

\newtheorem*{acknowledgements}{Acknowledgements}

\allowdisplaybreaks[3]

\def\Q{\mathbb Q}
\def\Z{\mathbb Z}

\def\O{\mathcal{O}}

\def\Spec{\operatorname{Spec}}
\def\Proj{\operatorname{Proj}}

\DeclareMathOperator{\ord}{ord}

\DeclareMathOperator{\length}{length}

\DeclareMathOperator{\WDiv}{WDiv}
\DeclareMathOperator{\CDiv}{CDiv}
\DeclareMathOperator{\Supp}{Supp}

\DeclareMathOperator{\Sign}{Sign}

\DeclareMathOperator{\Hom}{Hom}

\DeclareMathOperator{\Pic}{Pic}

\begin{document}
\title[Slope inequality, Morsification and moduli of curves]
{Slope inequality of fibered surfaces, Morsification conjecture and moduli of curves}
\author [M. Enokizono]{Makoto Enokizono}
\address{Graduate School of Mathematical Sciences, University of Tokyo, 3-8-1 Komaba, Meguro-ku, Tokyo 153-8914, Japan}
\email{\url{enokizono@g.ecc.u-tokyo.ac.jp}}


\begin{abstract}
Using the theory of moduli of curves, we establish various slope inequalities for general fibered surfaces.
More precisely, we introduce the notion of functorial divisors on Artin stacks and prove a theorem concerning their effectiveness.
Considering the above concept with the Morsification conjecture and the semistable reduction, 
we obtain several slope (in)equalities, e.g., a generalization of Moriwaki's slope inequality, slope equalities of general fibered surfaces whose fibers satisfy the Morsification conjecture.
As applications, we provide a positive answer to Reid's conjecture concerning algebraic Morsification of non-hyperelliptic fibrations of genus $3$, a positive partial answer to
the question posed by Lu and Tan regarding the Chern invariants of fiber germs,
and a partial result concerning lower bounds of the slope of effective divisors on the moduli spaces of stable curves.
\end{abstract}

\maketitle

\setcounter{tocdepth}{2}

\tableofcontents

\section{Introduction}
\label{sec:Introduction}

\subsection{Geography of algebraic surfaces and fibered surfaces}
One of the aim of algebraic geometry is to classify algebraic varieties.
Numerical invariants of algebraic varieties play a crucial role in the classification problem.
In the classification of minimal algebraic surfaces $S$ of general type, the numerical invariants $K_{S}^{2}$ and $\chi(\O_{S})$ have been traditionally used, and fundamental inequalities between them have been investigated, e.g.,  

\begin{itemize}
\item
(Noether \cite{Noe}) $K_{S}^{2}\ge 2\chi(\O_{S})-6$.

\item
(Castelnuovo \cite{Cas}) $K_{S}^{2}\ge 3\chi(\O_{S})-10$ when the canonical map of $S$ is birational onto its image.

\item
(Miyaoka--Yau \cite{Miy}, \cite{Yau}) $K_{S}^{2}\le 9\chi(\O_{S})$.
\end{itemize}

The first two inequalities are closely related to the slope inequalities of fibered surfaces. 
Indeed, it is known that almost all algebraic surfaces near the Noether line (resp.\ whose canonical maps are birational onto image near the Castelnuovo line) admit genus $2$ (resp.\ non-hyperelliptic genus $3$) fibrations $f\colon S\to \mathbb{P}^1$ (\cite{Hor}, \cite{AsKo}, \cite{Kon2}). In this case, the Noether inequality (resp.\ Castelnuovo inequality) is equivalent to the slope inequality $K_{f}^{2}\ge 2\chi_{f}$ (resp.\ $K_{f}^{2}\ge 3\chi_{f}$).
Here $f\colon S\to B$ is called a {\em fibered surface of genus $g$} or a {\em genus $g$ fibration} if $f$ is a surjective morphism from a smooth projective surface $S$ to a smooth projective curve $B$ whose general fiber is a smooth projective curve of genus $g$.
Let $K_{f}:=K_{S}-f^{*}K_{B}$ denote the relative canonical divisor of $f$ and consider two invariants $K_f^{2}$ and $\chi_{f}:=\deg f_{*}\O_{S}(K_{f})$.
The ratio $K_{f}^{2}/\chi_{f}$ is called the {\em slope} of $f$.
A {\em slope inequality} of relatively minimal fibered surfaces $f$ means an inequality of the form $K_{f}^{2}\ge a\chi_{f}$, where $a>0$ is a positive number which usually depends only on the geometric condition imposed on the general fibers.
In general, the well-known slope inequality 
\begin{equation} \label{usualslopeineq}
K_{f}^{2}\ge \frac{4(g-1)}{g}\chi_{f}
\end{equation}
holds for {\em any} relatively minimal fibered surface $f$ of genus $g\ge 2$ (\cite{Xia}, \cite{CoHa}).
However, to apply slope inequalities to the classification problem of surfaces of general type, sharper slope inequalities than \eqref{usualslopeineq} are needed under appropriate geometric conditions imposed on the general fibers.
To produce these slope inequalities, there are two different approaches.
The first one is to use the theory of algebraic surfaces (e.g., \cite{Xia}, \cite{Kon3}) and 
the second is to use the theory of moduli of curves (e.g., \cite{CoHa}, \cite{HaMo}).
On one hand, the advantage of using the surface theory is that it can handle degenerate fibers that are worse than semistable fibers.
However, this method relies on the geometric conditions imposed on  the general fibers and requires case-by-case studies.
On the other hand, using the moduli theory of curves provides some general results that do not depend on the conditions on the general fibers.
However, as long as using the moduli space of stable curves, 
only semistable fibers can be dealt with.
In this paper, we derive several slope inequalities for {\em general} fibered surfaces which allow to have non-semistable fibers.
To achieve this, we consider the moduli stacks of Gorenstein curves that are extensions of the moduli of stable curves and the intersection theory on them.

\subsection{Functorial divisors on Artin stacks}
One of the problem when using the extended moduli stacks is that they can be quite singular, may even fail to be separated, and may not admit a good moduli space.
Particularly, some tautological divisors defined by universal families may not be Cartier and so it may be careful to pullback by moduli maps.
Fortunately, certain tautological divisors computing numerical invariants of fibrations behave functorially with respect to flat pullbacks, alteration pushforwards and Gysin pullbacks.
Hence, instead of considering Cartier divisors, we introduce {\em functorial divisors} on Artin stacks.
Roughly speaking, a functorial divisor $D$ on an Artin stack $\mathcal{X}$ is a family of Weil divisors $\{D_{B}\}_{B}$ on an algebraic space $B$ over $\mathcal{X}$
which is compatible with flat pullbacks, alteration pushforwards and Gysin pullbacks (for the precise definition, see Definition~\ref{functorial}).
A functorial divisor $D$ determines a usual Weil divisor $D_{\mathcal{X}}$ by restricting it to the subfamily $\{D_{U}\}_{U}$ where $U$ is smooth over $\mathcal{X}$.
All Cartier divisors are regarded as functorial divisors.
Other important examples are as follows:
Let $\pi\colon \mathcal{Y}\to \mathcal{X}$ be a representable proper flat morphism of relative dimension $n$ between Artin stacks.
Then the pushforward by $\pi$ of the intersection of $n+1$ Cartier divisors $E_{1},\ldots, E_{n+1}$ on $\mathcal{Y}$ admits the structure of functorial divisors (Example~\ref{functdivex}~(2)).
A functorial divisor $D$ is called {\em effective} if $D_{B}$ is effective for any algebraic space $B$ over $\mathcal{X}$ such that each generic point of $B$ does not map to the support of $D_{\mathcal{X}}$.
The effectivity of $D$ clearly implies that the associated Weil divisor $D_{\mathcal{X}}$ is effective.
In this paper, we prove that the converse also holds:

\begin{theorem}[Theorem~\ref{effthm}] \label{Introeffthm}
Let $\mathcal{X}$ be a reduced and equidimensional Artin stack of finite type over a field $k$ that is regular in codimension one.
Then a functorial divisor $D$ on $\mathcal{X}$ is effective if and only if the associated Weil divisor $D_{\mathcal{X}}$ is effective.
\end{theorem}

The main idea to prove Theorem~\ref{Introeffthm} is to use the negativity lemma {\cite[Lemma~3.39]{KoMo}}.
Although Theorem~\ref{Introeffthm} is not necessarily essential for the slope inequalities discussed in this paper, the use of functorial divisors in the argument improves clarity.
It is expected that Theorem~\ref{Introeffthm} can be applied to other moduli stacks.

\subsection{Slope inequalities and the moduli of reduced curves}
In the rest of the section, we consider over the complex number field $\mathbb{C}$ for simplicity.
In this paper, we denote by $\mathcal{M}^{\mathrm{Gor}}_{g}$ (the normalization of) the moduli stack parametrizing smoothable, Gorenstein canonically polarized curves of genus $g\ge 2$.
This contains the moduli stack of stable curves $\overline{\mathcal{M}}_{g}$ as an open dense substack. 
The rational Picard group of $\overline{\mathcal{M}}_{g}$ is known to be freely generated by the Hodge bundle $\lambda$ and the boundary divisors $\delta_{0},\ldots,\delta_{\llcorner g/2 \lrcorner}$ when $g\ge 3$ (see \cite{ArCo}, \cite{Mor}).
Let $\kappa$ denote the first Morita--Mumford class on $\overline{\mathcal{M}}_{g}$,
which satisfies Noether's formula 
$$
12\lambda=\kappa+\sum_{i=0}^{\llcorner g/2 \lrcorner}\delta_{i}.
$$
By definition, these tautological divisors $\lambda$ and $\kappa$ are naturally extended as functorial divisors on $\mathcal{M}^{\mathrm{Gor}}_{g}$.
Let $f\colon S\to B$ be a relatively minimal fibered surface of genus $g$.
Then the relative canonical model $\overline{f}\colon \overline{S}\to B$ corresponds to a morphism $\rho\colon B\to \mathcal{M}^{\mathrm{Gor}}_{g}$.
Since the natural morphism $S\to \overline{S}$ is a minimal resolution of Du Val singularities, the invariants $K_{f}^{2}$ and $\chi_{f}$ coincide with $K_{\overline{f}}^{2}$ and $\chi_{\overline{f}}$, respectively.
Then we have 
$$
K_{f}^{2}=\deg \kappa_{B},\quad \chi_{f}=\deg \lambda_{B}.
$$
To establish slope inequalities, we want to show that $\kappa-a\lambda$ is effective as a functorial divisor for a suitable number $a>0$.
Even if $\kappa-a\lambda$ is effective on $\overline{\mathcal{M}}_{g}$, it is not immediately clear whether the divisor $\kappa_{B}-a\lambda_{B}$ is effective when the fibered surface $f\colon S\to B$ has non-semistable fibers.
However, applying Theorem~\ref{Introeffthm} reveals that it suffices to show the effectiveness of $\kappa-a\lambda$ as a Weil divisor on $\mathcal{M}^{\mathrm{Gor}}_{g}$.
For instance, this condition is satisfied if the complement $\mathcal{M}^{\mathrm{Gor}}_{g}\setminus \overline{\mathcal{M}}_{g}$ has codimension at least two. 
This illustrates a typical application of Theorem~\ref{Introeffthm}. 
However, determining whether this complement indeed has codimension at least two is difficult and not well understood, as it involves the deformation theory of non-reduced curves. 
As a test case, if one considers the open locus $\mathcal{M}^{\mathrm{lci+}}_{g}\subseteq \mathcal{M}^{\mathrm{Gor}}_{g}$ parametrizing reduced local complete intersection curves, the following holds:

\begin{theorem}[Theorem~\ref{codimbound}] \label{Introcodimbound}
The complement of $\overline{\mathcal{M}}_{g}$ in $\mathcal{M}^{\mathrm{lci+}}_{g}$ has no codimension one components.
\end{theorem}

Combining Theorem~\ref{Introeffthm} with Theorem~\ref{Introcodimbound}, we can show the following precise slope inequality, which is a generalization of Moriwaki's slope inequality {\cite[Theorem~D]{Mor2}}:

\begin{theorem}[Theorem~\ref{Moriwakiineq}] \label{IntroMoriwakiineq}
Let $f\colon S\to B$ be a relatively minimal fibered surface of genus $g$.
Assume that all the fibers of the relative canonical model of $f$ are reduced.
Then we have
$$
K_{f}^{2}\ge \frac{4(g-1)}{g}\chi_{f}+\sum_{i=1}^{\llcorner g/2 \lrcorner}\frac{4i(g-i)-g}{g}\sum_{p\in B}\delta_{i}(f^{-1}(p)),
$$
where $\delta_{i}(f^{-1}(p))$ is the virtual number of type $i$ nodes in the fiber $f^{-1}(p)$, see Example~\ref{localinvex}~(4).
\end{theorem}


\subsection{Morsification conjecture and Moduli conjecture}

A classical theorem in Morse theory states that any smooth function on a smooth manifold can be perturbed slightly to become a Morse function. 
One may ask for an analogue of this statement in the category of complex manifolds and holomorphic maps.
To begin with, consider a complex surface equipped with a proper holomorphic fibration $f\colon S\to B$.
In this setting, the counterpart of a Morse function is given by a Lefschetz fibration, that is, a family of stable curves with at most one node in each fiber. 
The perturbation of a function corresponds to a {\em splitting deformation} of a fiber germ $f\colon S\to (p\in B)$ (Definition~\ref{splittingdef}).
Among singular fibers, those which cannot be further decomposed (called {\em atomic} fibers) by splitting deformations consist not only of stable fibers with a single node, but also of smooth multiple fibers (\cite{Tak}). 
The Morsification conjecture, proposed by Xiao and Reid, asserts that these exhaust all possible atomic fibers:
\begin{conjecture}[Conjecture~\ref{Morsconj}, Morsification conjecture]
A fiber germ $f^{-1}(p)$ is atomic if and only if it is a stable curve with one node or a multiple of a smooth curve.
\end{conjecture}
The if part is known but the only if part is only known for $g\le 6$ (\cite{Hor3}, \cite{TakIII}, \cite{Oku}).
A relatively minimal fiber germ $f\colon S\to (p\in B)$ {\em satisfies the Morsification conjecture} if it is splittable into stable fibers with one node and smooth multiple fibers by finitely many splitting deformations (Definition~\ref{splittable}).
By using Theorem~\ref{Introcodimbound}, any relatively minimal fiber germ whose canonical model has the reduced central fiber satisfies the Morsification conjecture (Proposition~\ref{splitcri}).

On the other hand, the moduli stack $\mathcal{M}^{\mathrm{Gor}}_{g}$ may have divisorial components of the complement of $\overline{\mathcal{M}}_{g}$.
If there is a divisorial component $\Delta\subseteq \mathcal{M}^{\mathrm{Gor}}_{g}\setminus \overline{\mathcal{M}}_{g}$, then a fiber germ $f\colon S\to (p\in B)$ whose central fiber corresponding to a general point of $\Delta$ can not be splittable into at least two singular fibers, that is, it is atomic (Remark~\ref{relMorsMod}~(4)).
So, the Morsification conjecture suggests that any codimension one component of $\mathcal{M}^{\mathrm{Gor}}_{g}\setminus \overline{\mathcal{M}}_{g}$ generically parametrizes smooth multiple curves appearing in fibers of relatively minimal fibered surfaces or singular curves that can not appear in fibers of the relative canonical models of fibered surfaces.
To connect the theory of fibered surfaces with the theory of moduli of curves, we propose the following conjecture regarding the moduli of curves:
\begin{conjecture}[Conjecture~\ref{moduliconj}, Moduli conjecture]
There exists a moduli stack $\mathcal{M}^{\star}_{g}$ of curves which contains all curves of genus $g$ appearing in fibers of the relative canonical model of fibered surfaces as geometric points such that any divisorial component of $\mathcal{M}^{\star}_{g}\setminus \overline{\mathcal{M}}_{g}$ generically parametrizes smooth multiple curves appearing in fibers of relatively minimal fibered surfaces.
\end{conjecture}

\subsection{Slope inequalities and divisors on the moduli stack}
The problem of the effectiveness of $\kappa-a\lambda$ is related to finding effective divisors $D$ on $\overline{\mathcal{M}}_{g}$ with small slope.
For an effective divisor $D$ on $\overline{\mathcal{M}}_{g}$ that contains no boundary divisors, the {\em slope} of $D$ is defined by $s_{D}:=a/b$, where we write 
$$
D=a\lambda-\sum_{i=0}^{\llcorner \frac{g}{2} \lrcorner}b_{i}\delta_{i}
$$
 in the Picard group of $\overline{\mathcal{M}}_{g}$ and $b$ is the minimum of $b_{i}$'s.
It follows that the $\Q$-divisor $\kappa-(12-s_{D})\lambda$ on $\overline{\mathcal{M}}_{g}$ can be written by a non-negative linear combination of $D$ and $\delta_{i}$'s.
In particular, it is effective as a Weil divisor on $\overline{\mathcal{M}}_{g}$.
Moreover, $\kappa-(12-s_{D})\lambda$ can be considered as a functorial divisor on $\mathcal{M}^{\mathrm{Gor}}_{g}$ the restriction of which on $\overline{\mathcal{M}}_{g}$ equals the above effective divisor.
This functorial divisor is referred as the {\em Horikawa divisor associated to $D$} and denoted by $H_{D}$ (Definition~\ref{efftautodef}).

A fiber germ or a fibered surface of genus $g$ is called {\em $D$-general} if its general fiber is not contained in $D$ as a moduli point.
For a relatively minimal $D$-general fiber germ $f\colon S\to (p\in B)$, the {\em Horikawa index} $H_{D}(f^{-1}(p))$ is defined by the degree of the divisor $(H_{D})_{B}$ (Example~\ref{localinvex}~(1)), where $(H_{D})_{B}$ is the trace of the functorial divisor $H_{D}$ at the moduli map $\rho\colon B\to \mathcal{M}^{\mathrm{Gor}}_{g}$ of the relative canonical model of $f$.
When $\mathcal{M}^{\mathrm{Gor}}_{g}\setminus \overline{\mathcal{M}}_{g}$ has divisorial components, it is not known whether $H_{D}$ is effective on $\mathcal{M}^{\mathrm{Gor}}_{g}$. 
In particular, it is unclear whether the Horikawa index $H_{D}(f^{-1}(p))$ is non-negative.
The main theorem in this paper says that the Horikawa index is non-negative when assuming the Morsification conjecture or the Moduli conjecture:

\begin{theorem}[Theorem~\ref{slopeeq}, Slope equality] \label{Introslopeeq}
Let $D$ be an effective divisor on $\overline{\mathcal{M}}_{g}$ which contains no boundary divisors $($or $D=\delta_{1}$ and $g=2$$)$ with $s_{D}\ge 4$.
Let $f\colon S\to (p\in B)$ be a relatively minimal $D$-general fiber germ of genus $g$
 and assume that it satisfies the Morsification conjecture or the Moduli conjecture holds true.
Then the Horikawa index $H_{D}(f^{-1}(p))$ is non-negative.
In particular, for any relatively minimal $D$-general fibered surface $f\colon S\to B$ of genus $g$ assuming that all fibers of $f$ satisfy the Morsification conjecture or the Moduli conjecture holds true, we have 
$$
K_{f}^{2}=(12-s_{D})\chi_{f}+\sum_{p\in B}H_{D}(f^{-1}(p))\ge (12-s_{D})\chi_{f}.
$$
\end{theorem}

This result is stronger than the slope inequality $K_{f}^{2}\ge (12-s_{D})\chi_{f}$, in that the difference from the lower bound of the slope is expressed as a sum of local contributions given by the Horikawa index. 
An identity of this form is referred to as a {\em slope equality} (see the survey \cite{AsKo2}).
As a corollary, the slope equality holds for $D$-general fibered surfaces whose relative canonical models have reduced fibers.

The proof of Theorem~\ref{Introslopeeq} reduces to showing that $H_{D}(f^{-1}(p))$ is non-negative for any smooth multiple fiber germ $f^{-1}(p)$.
In this case, the non-negativity of $H_{D}(f^{-1}(p))$ is shown by using semistable reduction.
The invariance of Horikawa indices under splitting deformations (Proposition~\ref{locinvsplit}) and  
the behavior of Horikawa indices under semistable reductions (Proposition~\ref{Horindcompare}) play an important role in the proof.

Analyzing the difference from the lower bound of the slope of fibered surfaces is important in the classification theory of surfaces of general type.
Indeed, motivated by classifying algebraic surfaces near the Noether line, Horikawa \cite{Hor2} introduced non-negative invariants $H(f^{-1}(p))$ for genus $2$ fiber germs $f^{-1}(p)$, and showed that the slope equality 
$$
K_{f}^{2}=2\chi_{f}+\sum_{p\in B}H(f^{-1}(p))
$$ 
holds for any relatively minimal fibered surface $f\colon S\to B$ of genus $2$.
The term ``Horikawa index'' for $H_{D}(f^{-1}(p))$ originates from this result.
Moreover, Horikawa \cite{Hor3} also showed that $H(f^{-1}(p))$ equals the number of type $1$ nodes on the stable fibers split from $f^{-1}(p)$.
Later, Reid \cite{Rei} established a similar result for non-hyperelliptic genus $3$ fibrations, and conjectured that the Horikawa indices can be decomposed into three types according to ``algebraically atomic'' fibers ({\cite[Conjecture~3.2]{Rei}}, {\cite[p.33, Conjecture]{AsKo}}).
Assuming the Morsification conjecture, we can show the following decomposition theorem of Horikawa indices:

\begin{theorem}[Corollary~\ref{Hordecomp}] \label{IntroHordecomp}
Let $D=\sum_{i}m_{i}D_i$ be the irreducible decomposition.
Let $f\colon S\to (p\in B)$ be a $D$-general fiber germ of genus $g$ that satisfies the Morsification conjecture.
Then, there exist finite non-negative numbers $n_1,\ldots, n_{l}$ depending only on $D$ such that $f^{-1}(p)$ is splittable by finitely many splitting deformations into fiber germs $\bm{f}_{t}^{-1}(p_{t})$ of the following types:

\smallskip

\noindent 
$(\mathrm{i})$
Smooth fiber germs which do not belong to $D$ as moduli points.
In this case, 
$$
H_{D}(\bm{f}_{t}^{-1}(p_{t}))=0.
$$

\smallskip

\noindent 
$(\mathrm{ii})$
Smooth fiber germs whose moduli maps meet the support of $D$ transversely at a smooth point of $D_i$.
In this case, 
$$
H_{D}(\bm{f}_{t}^{-1}(p_{t}))=\frac{m_{i}}{b}.
$$

\smallskip

\noindent 
$(\mathrm{iii})$
Stable fiber germs with one node of type $i$ which do not belong to $D$ as moduli points.
In this case, 
$$
H_{D}(\bm{f}_{t}^{-1}(p_{t}))=\frac{b_{i}-b}{b}.
$$

\smallskip

\noindent 
$(\mathrm{iv})$
$D$-indecomposable smooth multiple fiber germs of type $(m, n_{i})$ as in Definition~\ref{smmulttype}.
In this case, 
$$
H_{D}(\bm{f}_{t}^{-1}(p_{t}))=\frac{n_{i}}{mb}+\frac{s_{D}-4}{2}\left(1-\frac{1}{m} \right)(g-1).
$$
\end{theorem}

Since the Morsification conjecture is true for genus $g\le 6$,
Theorems~\ref{Introslopeeq} and \ref{IntroHordecomp} give an affirmative answer to Reid's conjecture (Example~\ref{genus2345}~(2)).

By Theorem~\ref{Introslopeeq}, to establish better slope inequalities for general fibered surfaces, it suffices to find an effective divisor $D$ with smaller slope $s_{D}$.
Motivated by the birational geometry of $\overline{\mathcal{M}}_{g}$, there have been many studies to find effective divisors with small slopes (e.g., \cite{HaMu}, \cite{EiHa}, \cite{Far}).
Calculating the slope of $D$ for the purpose of establishing the slope inequality seems somewhat tautological.
However, to compute the slope of $D$, it is sufficient to solve a system of equations regarding the coefficients of $D$, which is determined by finitely many specific semistable fibered surfaces.
Thus, the results of this paper suggest that {\em due to the existence and algebraicity of the moduli stack,
some numerical properties for fibrations can be determined by finitely many specific semistable fibrations}.
It is expected that similar techniques using moduli stacks can also be applied to study the modular invariants of families of other polarized varieties with certain properties (e.g.,  trigonal fibered surfaces, families of $\Q$-Gorenstein curves, polarized K3 fibrations, and so on).

If $g$ is odd and $D$ is the Brill--Noether locus parametrizing curves with non-maximal gonality, then an explicit presentation of $D$ in the Picard group of $\overline{\mathcal{M}}_{g}$ was given in \cite{HaMu}.
In this case, Theorem~\ref{Introslopeeq} recovers Konno's slope (in)equality {\cite[Theorem~4.1]{Kon}} under the assumption of the Morsification conjecture or the Moduli conjecture (Corollary~\ref{BNslopeineq}).
If $g$ is even and $D$ is the Gieseker--Petri locus parametrizing curves with pencils of degree $g/2+1$ which violates the Petri condition, then an explicit presentation of $D$ in the Picard group of $\overline{\mathcal{M}}_{g}$ was given in \cite{EiHa}, and Theorem~\ref{Introslopeeq} give a slope (in)equality for general fibered surfaces of even genus $g$ under the assumption that the Morsification conjecture or the Moduli conjecture (Corollary~\ref{GPslopeineq}).
Combining this with the recent result of Okuda \cite{Oku} showing the Morsification conjecture for $g=6$,
we establish the slope equality for general fibered surfaces of genus $6$:
\begin{corollary}[Corollary~\ref{g6slopeineq}] \label{Introg6slopeineq}
Let $f\colon S\to B$ be a relatively minimal fibered surface of genus $6$ such that  general fibers are tetragonal and any gonality pencil satisfies the Petri condition.
Then, we have
$$
K_{f}^{2}=\frac{25}{6}\chi_{f}+\sum_{p\in B}H_{GP_{6,4}^{1}}(f^{-1}(p)) \ge \frac{25}{6}\chi_{f}.
$$
\end{corollary}
This is already a new result even when restricted to the slope inequality.

\subsection{Topological aspects}
Let $S$ be a smooth projective surface. 
When regarded as a smooth $4$-dimensional real manifold, two important topological invariants of $S$ are the topological Euler characteristic $\chi_{\mathrm{top}}(S)$ and the signature $\Sign(S)$. 
Giving a rational function on $S$ roughly corresponds, via blow-ups and the Stein factorization, to giving a fibered surface $f\colon S\to B$.
Classical Morse theory shows that a Morse function on a manifold encodes rich topological information about the manifold. 
In the present setting, the Morsification conjecture predicts that $f$ can locally be morsified, with respect to the base curve $B$, by means of splitting deformations. 
It is therefore natural to ask whether the topological invariants of $S$ can be expressed as sums of local invariants associated with the fibers of $f$, in such a way that these local invariants behave additively under splitting deformations.

For example, the topological Euler contribution $\delta(f^{-1}(p))$ (Example~\ref{localinvex}~(3)) provides such a local invariant, yielding a decomposition of the relative topological Euler characteristic
$$
e_{f}:=\chi_{\mathrm{top}}(S)-\chi_{\mathrm{top}}(F)\chi_{\mathrm{top}}(B)=\sum_{p\in B}\delta(f^{-1}(p)),
$$
where $F$ is a general fiber of $f$.
On the other hand, the analogous problem for the signature has been extensively studied from both algebraic-geometric and topological viewpoints (see \cite{AsKo2}, \cite{Kun}). 
In general, however, local signatures localizing the signature cannot be constructed canonically, nor are they unique, except in special cases such as genus $g \le 2$ or hyperelliptic fibrations.

The method developed in this paper provides a different approach. 
Given an effective divisor $D$ on $\overline{\mathcal{M}}_{g}$ containing no boundary divisors, one can define a local signature $\sigma_{D}(f^{-1}(p))$ (Example~\ref{localinvex}~(2)) for $D$-general fiber germs $f^{-1}(p)$ as in the definition of the Horikawa index replacing $\kappa-(12-s_D)\lambda$ with $\kappa-8\lambda$, satisfying
$$
\Sign(S)=\sum_{p\in B}\sigma_{D}(f^{-1}(p))
$$
for any $D$-general fibered surface $f\colon S\to B$ of genus $g$.
Moreover, it is almost immediate from the definition that this local signature behaves additively under splitting deformations. 
The relation among the local signature, the Horikawa index, and the topological Euler characteristic is as follows:
$$
\sigma_{D}(f^{-1}(p))=\frac{1}{s_D}\left(4H_{D}(f^{-1}(p))-(s_D-4)\delta(f^{-1}(p)) \right).
$$
Thus, Theorem~\ref{IntroHordecomp} also yields the decomposition theorem for local signatures.

\subsection{Applications}
For a relatively minimal fiber germ $f\colon S\to (p\in B)$ of genus $g$, its {\em Chern invariants} $c_{1}^{2}(f^{-1}(p))$ and $\chi_{f^{-1}(p)}$ are introduced in \cite{TanII} as non-negative rational numbers measuring the non-semistability of $f^{-1}(p)$ (Definition~\ref{Cherndef}).
These invariants are completely determined by the topology of a fiber germ $f^{-1}(p)$ (Remark~\ref{Chernrem}).
The question posed by Lu and Tan {\cite[p.3395, Questions (3)]{LuTa}} asks whether $c_{1}^{2}(f^{-1}(p))\ge \chi_{f^{-1}(p)}$ holds.
As an unexpected application of Theorem~\ref{Introslopeeq}, we give an affirmative partial answer to this question:

\begin{theorem}[Theorem~\ref{LuTanQ}]  \label{IntroLuTanQ}
Let $f\colon S\to (p\in B)$ be a relatively minimal fiber germ of genus $g$.
Assume that it satisfies the Morsification conjecture or the Moduli conjecture holds true.
Then we have $c_{1}^2(f^{-1}(p))\ge \chi_{f^{-1}(p)}$.
\end{theorem}

After completing this result, we learned from Cheng and Lu \cite{ChLu} that Lu--Tan's question has been answered affirmatively. 
Our method is completely different from theirs.
Indeed, in our proof, we apply Theorem~\ref{Introslopeeq} to general ample divisors $D$ of the form $D=a\lambda-b\delta$, and then use the Cornalba--Harris ampleness criterion {\cite[Theorem~(1.3)]{CoHa}} $s_{D}>11$.
It is somewhat surprising that the ampleness criterion $s_{D}>11$ corresponds to the slope inequality for the Chern invariants $c_{1}^2(f^{-1}(p))\ge \chi_{f^{-1}(p)}$.

Another application is in relation to the lower bound of the slope of effective divisors on $\overline{\mathcal{M}}_{g}$.
Let $s_{g}$ denote the infimum of the slopes $s_{D}$ of effective divisors $D$ on $\overline{\mathcal{M}}_{g}$ which contain no boundary divisors.
In \cite{CFM}, Chen conjectured that $\lim_{g\to \infty}s_{g}=0$ while Farkas and Morrison conjectured that $\lim_{g\to \infty}s_{g}=6$.
The following applications of Theorem~\ref{Introslopeeq} support the later conjecture:

\begin{theorem}[Theorem~\ref{slope4}] \label{Introslope4}
Assume that the Morsification conjecture or the Moduli conjecture holds true.
Let $D$ be an effective divisor on $\overline{\mathcal{M}}_{g}$ which contains no boundary divisors.
If $s_{D}\le 4$, then $D$ contains the locus of smooth curves with non-trivial automorphisms with fixed points.
\end{theorem}

\begin{theorem}[Theorem~\ref{slope6}] \label{Introslope6}
Let $D$ be an effective divisor on $\overline{\mathcal{M}}_{g}$ which contains no boundary divisors.
If $s_{D}< 6$, then $D$ contains the locus of smooth curves with automorphisms of reduced type as in Definition~\ref{reducedtypedef}.
\end{theorem}

These theorems are also proved by the same principle as above, namely, from the correspondence between the inequality for the slope $s_{D}$ and a Miyaoka--Yau type inequality giving an upper bound for the slope of the Chern invariants ({\cite[Theorem~3.4, Proposition~3.5]{TanII}}). 

The above three applications come from the positivity of Horikawa indices (Theorem~\ref{slopeeq}) and the formula of Horikawa indices under semistable reductions (Proposition~\ref{Horindcompare}).
This means that the Chern invariants of fiber germs are subject to  moduli theoretic effects although they are completely determined by the topological types of fiber germs.
It seems to suggest a phenomenon that can be called a {\em fibration version of the McKay correspondence}.



\subsection{Structure of the paper}
The present paper is organized as follows.
In Section~\ref{sec:Intersection theory on Artin stacks},
we introduce functorial divisors on Artin stacks and prove Theorem~\ref{Introeffthm}.
In Section~\ref{sec:Moduli of curves}, we introduce the various moduli stacks of Gorenstein canonically polarized curves used in this paper.
We also discuss the basepoint-freeness and very-ampleness of pluricanonical linear systems on Gorenstein canonically polarized curves, and prove Theorem~\ref{Introcodimbound}.
In Section~\ref{sec:Morsification conjecture}, we discuss the relationship between the Morsification conjecture and the moduli of curves, and propose the Moduli conjecture.
In Section~\ref{sec:Slope inequality of fibered surfaces},
we first introduce Horikawa divisors.
Using the Horikawa divisors on the moduli stack of curves,
 various slope (in)equalities are established (Theorem~\ref{Introslopeeq}).
It also offers a generalization of Moriwaki's slope inequality (Theorem~\ref{IntroMoriwakiineq}), the decomposition theorem of Horikawa indices (Theorem~\ref{IntroHordecomp}), an affirmative partial answer of the question posed by Lu--Tan (Theorem~\ref{IntroLuTanQ}) and applications to the lower bound of the slope of effective divisors on $\overline{\mathcal{M}}_{g}$ (Theorems~\ref{Introslope4} and \ref{Introslope6}).

\begin{acknowledgements}
The author would like to express his sincere gratitude to Kazuhiro Konno and Tadashi Ashikaga for many valuable discussions and their warm encouragement. 
He is also grateful to Kazuhiko Yamaki and Kentaro Mitsui for helpful conversations on the moduli stack of curves. 
The author thanks Jun Lu for bringing to his attention the preprint \cite{ChLu}, which gives an affirmative answer to the Lu--Tan question, and Takayuki Okuda for informing him of the preprint \cite{Oku}, which resolves the Morsification conjecture in genus $6$.
The author also thanks Gert-Martin Greuel and Gavril Farkas for answering questions.
He was partially supported by JSPS KAKENHI No.20K14297 and No.25K06926. 
\end{acknowledgements}

\section{Intersection theory on Artin stacks}
\label{sec:Intersection theory on Artin stacks}

In this section, we recall the intersection theory on Artin stacks (cf.\ \cite{Vis}, \cite{Kre}).
For basic notions for Artin stacks, see for example \cite{Sta}, \cite{Ols}. 
Note that for many properties $\mathcal{P}$ of schemes such as being reduced, normal, equidimensional, of finite type over a base scheme, an Artin stack $\mathcal{X}$ has the property $\mathcal{P}$ 
if and only if there exists a smooth surjective morphism $U\to \mathcal{X}$ from a scheme $U$ with the property $\mathcal{P}$. 
We call an Artin stack $\mathcal{X}$ {\em irreducible} (resp.\ {\em integral}) if the underlying topological space $|\mathcal{X}|$ is not the union of two proper closed subsets (resp.\ $\mathcal{X}$ is irreducible and reduced).
Throughout this section, we assume that $\mathcal{X}$ is an Artin stack of finite type over a base field $k$,
and all schemes and algebraic spaces are also assumed to be of finite type over $k$.

\subsection{Cycles and divisors}
For an integer $d$, a {\em $d$-cycle} $\sum_{i}n_i\mathcal{Z}_i$ on $\mathcal{X}$ is a finite formal linear combination of $d$-dimensional integral closed substacks $\mathcal{Z}_i$ of $\mathcal{X}$. 
For a detailed exposition on the dimension theory of Artin stacks, see for example \cite[Tag 0AFL, Tag 0DRE]{Sta}.
We denote by $Z_{d}(\mathcal{X})$ the free abelian group consisting of $d$-cycles on $\mathcal{X}$.
If $\mathcal{X}$ is reduced and of equidimension $n$, then {\em Weil divisors} on $\mathcal{X}$ are $(n-1)$-cycles on $\mathcal{X}$ and we denote the group of Weil divisors on $\mathcal{X}$ by 
$$
\WDiv(\mathcal{X}):=Z_{n-1}(\mathcal{X}).
$$
A {\em prime divisor} on $\mathcal{X}$ means an integral closed substack of $\mathcal{X}$ of codimension $1$.

A {\em rational function} $f$ on a reduced and equidimensional stack $\mathcal{X}$ is defined as a morphism $f\colon \mathcal{U}\to \mathbb{A}^1$, where $\mathcal{U}$ is an open dense substack of $\mathcal{X}$ and $\mathbb{A}^1$ denotes the affine line.
The set of rational functions on $\mathcal{X}$ forms a ring, specifically, a finite product of fields, denoted as $\mathrm{Rat}(\mathcal{X})$.
For an invertible rational function $f\in \mathrm{Rat}(\mathcal{X})^{*}$ and a prime divisor $\mathcal{Y}$ on $\mathcal{X}$, we define the {\em order of $f$ along $\mathcal{Y}$} as follows:
If $\mathcal{X}$ is irreducible, then we define
$$
\ord_{\mathcal{Y}}(f):=\ord_{\mathcal{Y}\times_{\mathcal{X}}U}(f\circ \pi),
$$
where $\pi\colon U\to \mathcal{X}$ is a smooth morphism from an integral scheme $U$ so that $\mathcal{Y}\times_{\mathcal{X}}U$ is a prime divisor on $U$
and $f \circ \pi$ is regarded as a rational function on $U$.
In general, let $\mathcal{X}_{i}$, $i=1,\ldots,l$ be the irreducible components of $\mathcal{X}$ containing $\mathcal{Y}$ and put
$$
\ord_{\mathcal{Y}}(f):=\sum_{i=1}^{l}\ord_{\mathcal{Y}}(f|_{\mathcal{X}_{i}}).
$$
The divisor 
$$
\mathrm{div}_{\mathcal{X}}(f):=\sum_{\text{$\mathcal{Y}$: prime divisor}}\ord_{\mathcal{Y}}(f)\mathcal{Y}
$$
is called the {\em principal divisor on $\mathcal{X}$ associated to $f$}.

For an integer $d$, we define the {\em group of rational equivalences of dimension $d$} by
$$
W_{d}(\mathcal{X}):=\bigoplus_{\mathcal{Z}\subseteq \mathcal{X}}\mathrm{Rat}(\mathcal{Z})^{*},
$$
where $\mathcal{Z}\subseteq \mathcal{X}$ runs through all integral closed substacks of dimension $d+1$.
We define a homomorphism 
$$
\partial\colon W_{d}(\mathcal{X})\to Z_{d}(\mathcal{X})
$$ 
as follows:
We denote by $(\mathcal{Z}, f)$ an element $f$ of the direct summand $\mathrm{Rat}(\mathcal{Z})^{*}$ of $W_{d}(\mathcal{X})$.
Then we put $\partial((\mathcal{Z}, f)):=\mathrm{div}_{\mathcal{Z}}(f)$ and regard it as a cycle on $\mathcal{X}$.
We define the {\em $d$-th naive Chow group of $\mathcal{X}$}
 by 
$$
A^{\circ}_{d}(\mathcal{X}):=Z_{d}(\mathcal{X})/\partial W_{d}(\mathcal{X}).
$$
When $\mathcal{X}$ is represented by an algebraic space $X$, we write $A_{d}(X)$ instead of $A^{\circ}_{d}(\mathcal{X})$.
In this paper, the $d$-th Chow group $A_{d}(\mathcal{X})$ introduced in \cite{Kre} is not used.

\begin{remark}
The group of cycles $Z_{*}(\mathcal{X})=\bigoplus_{d}Z_{d}(\mathcal{X})$ (resp.\ the group of rational equivalences $W_{*}(\mathcal{X})=\bigoplus_{d}W_{d}(\mathcal{X})$) can be regarded as the global section of the sheaf of cycles $\mathcal{Z}_{*}(\mathcal{X})$ (resp.\ the sheaf of rational equivalences $\mathcal{W}_{*}(\mathcal{X})$) on the lisse-\'{e}tale site of $\mathcal{X}$.
Namely, any cycle $\alpha=\sum_{i}n_i\mathcal{Z}_i$ on $\mathcal{X}$ has a presentation $$
\alpha=\{\alpha_{U}\}_{U\to \mathcal{X}},
$$
where $U\to \mathcal{X}$ runs through all smooth morphisms from schemes $U$ and
$\alpha_{U}=\sum_{i}n_i\mathcal{Z}_{i, U}$ is a cycle on $U$.
However, the same is not true for the naive Chow group $A^{\circ}_{*}(\mathcal{X})=\bigoplus_{d}A^{\circ}_{d}(\mathcal{X})$
since the Chow groups on the lisse-\'{e}tale site of $\mathcal{X}$ are not sheaves in general.
\end{remark}

Let $\mathcal{X}$ be a reduced stack of equidimension $n$.
A {\em Cartier divisor} on $\mathcal{X}$ is a global section of the sheaf $\mathcal{K}^{*}_{\mathcal{X}}/\O^{*}_{\mathcal{X}}$ on the lisse-\'{e}tale site of $\mathcal{X}$, where $\mathcal{K}_{\mathcal{X}}$ is the sheaf of rational functions: $\mathcal{K}_{\mathcal{X}}(U)=\mathrm{Rat}(U)$ for a smooth morphism $U\to \mathcal{X}$ from an algebraic space $U$.
It can be denoted by a data 
$$
\{(U_i, f_i)\}_{i}
$$
where $\{U_i\to \mathcal{X}\}_{i}$ is a smooth covering from algebraic spaces $U_i$ and $f_{i}$ is a rational function on $U_i$ such that $f_{i}/f_{j}|_{U_{ij}}$ is an invertible regular function on the algebraic space $U_{ij}$ representing $U_{i}\times_{\mathcal{X}}U_{j}$.
The group of Cartier divisors on $\mathcal{X}$ is denoted by 
$$
\CDiv(\mathcal{X}):=\Gamma(\mathcal{X}, \mathcal{K}^{*}_{\mathcal{X}}/\O^{*}_{\mathcal{X}}).
$$
For a Cartier divisor $D=\{(U_i, f_i)\}$, the principal divisors $\mathrm{div}_{U_{i}}(f_i)$ on $U_i$ form a global section of $\mathcal{Z}_{n-1}(\mathcal{X})$ and then corresponds to a Weil divisor on $\mathcal{X}$, which is denoted by $D^{\mathrm{W}}$.
This correspondence defines a natural homomorphism
$$
\CDiv(\mathcal{X})\to \WDiv(\mathcal{X});\quad D\mapsto D^{\mathrm{W}},
$$
which is injective when $\mathcal{X}$ is normal.

As in \cite[Definition~2.2.1]{Ful}, we define a {\em pseudo-divisor} on $\mathcal{X}$ as a triple 
$$
(L, Z, s)
$$
where
$L$ is an invertible $\O_{\mathcal{X}}$-module on the lisse-\'{e}tale site of $\mathcal{X}$, $Z$ is a closed subset of $|\mathcal{X}|$ and $s$ is an isomorphism $s\colon \O_{\mathcal{X}\setminus Z}\cong L|_{\mathcal{X}\setminus Z}$ over the complement of $Z$.
For a Cartier divisor $D=\{(U_i, f_i)\}$, the triple 
$$
(\O_{\mathcal{X}}(D), \Supp D, s_{D})
$$
forms a pseudo-divisor, where the sheaf $\O_{\mathcal{X}}(D)$ is defined by gluing the sheaf $\O_{U_{i}}f_{i}^{-1}$ on $U_i$ canonically and $s_{D}$ is the natural isomorphism defined by gluing the identity $\O_{U_{i}\setminus \Supp D}=\O_{U_{i}\setminus \Supp D}f_{i}^{-1}$.
The converse also holds:

\begin{lemma}[{\cite[Lemma~2.2]{Ful}}] \label{pdivCar}
Let $(L, Z, s)$ be a pseudo-divisor on $\mathcal{X}$ such that $Z$ is nowhere dense in $\mathcal{X}$.
Then there exists a Cartier divisor $D$ uniquely which represents $(L, Z, s)$, that is, $\Supp D\subseteq Z$ and there is an isomorphism $\O_{\mathcal{X}}(D)\cong L$ which sends $s_{D}$ to $s$.
\end{lemma}

\begin{proof}
We take a smooth covering $\{U_{\alpha}\to \mathcal{X}\}_{\alpha}$ from schemes $U_{\alpha}$ and trivializations $\varphi_{\alpha}\colon \O_{U_{\alpha}}\cong L|_{U_{\alpha}}$,
and we put $U_{\alpha_{0}}:=\mathcal{X}\setminus Z$ and $\varphi_{\alpha_{0}}:=s$.
The transition functions are isomorphisms 
$$
g_{\alpha \beta}:=\varphi_{\alpha}^{-1}\circ \varphi_{\beta}\colon \O_{U_{\alpha}}|_{U_{\alpha \beta}}\cong L|_{U_{\alpha \beta}} \cong\O_{U_{\beta}}|_{U_{\alpha \beta}},
$$
where $U_{\alpha \beta}$ is the algebraic space representing $U_{\alpha}\times_{\mathcal{X}}U_{\beta}$.
Note that $U_{\alpha \alpha_{0}}\subseteq U_{\alpha}$ is an open dense subscheme for each $\alpha$.
Let $f_{\alpha}:=g_{\alpha \alpha_{0}}\in \Gamma(\O_{U_{\alpha \alpha_{0}}}^{*})$ and regard it as a rational function on $U_{\alpha}$.
Since $f_{\alpha}/f_{\beta}=g_{\alpha \beta}$ for any $\alpha$ and $\beta$, it follows that $D:=\{(U_{\alpha}, f_{\alpha})\}_{\alpha}$ defines a Cartier divisor.
Since $L$ and $\O_{\mathcal{X}}(D)$ have the same transition functions $\{g_{\alpha \beta}\}$, there is a natural isomorphism $\O_{\mathcal{X}}(D)\cong L$ which sends $s_{D}$ to $s$.
\end{proof}

If $\mathcal{X}$ is regular, then any Weil divisor $D$ comes from a unique Cartier divisor, that is, $\CDiv(\mathcal{X})\cong \WDiv(\mathcal{X})$.
Indeed, the Weil divisor $D$ defines the invertible $\O_{\mathcal{X}}$-module $\O_{\mathcal{X}}(D)$ defined by $\O_{U}(D_{U})$ for any smooth morphism $U\to \mathcal{X}$ from a scheme $U$, 
and so defines a pseudo-divisor $(\O_{\mathcal{X}}(D), \Supp D, s_{D})$.
By Lemma~\ref{pdivCar}, it corresponds to a Cartier divisor.

For a morphism $f\colon \mathcal{Y}\to \mathcal{X}$ between Artin stacks and $D=(L, Z, s)$ a pseudo-divisor on $\mathcal{X}$,
we define the {\em pullback} $f^{*}D$ of $D$ by $f$ as the pseudo-divisor 
$$
(f^{*}L, f^{-1}(Z), f^{*}s)
$$
on $\mathcal{Y}$.
By using this, the pullback $f^{*}D$ of a Cartier divisor $D$ on $\mathcal{X}$ is uniquely determined if $f^{-1}(Z)$ is nowhere dense.
If $f^{-1}(Z)$ is not nowhere dense and $\mathcal{Y}$ has an open dense subscheme (e.g., $\mathcal{Y}$ is a Zariski locally quasi-separated algebraic space. cf.\ {\cite[Theorem~6.4.1]{Ols}}), the invertible sheaf $f^{*}L$ is trivial on an open dense subscheme and so determined as a Cartier divisor up to linear equivalence by the same proof of Lemma~\ref{pdivCar}.

\subsection{Pullback and pushforward of cycles}

Let $f\colon X\to Y$ be a proper (resp.\ flat) morphism of algebraic spaces.
For any cycle $\alpha$ on $X$ (resp.\ on $Y$), the {\em proper pushforward} $f_{*}\alpha$ (resp.\ {\em flat pullback} $f^{*}\alpha$) by $f$ is defined similarly as the scheme case (\cite[Definition~(3,6)]{Vis}).

Let $f\colon X\hookrightarrow Y$ be a regular closed immersion of codimension $c$ between algebraic spaces, that is, a closed immersion such that the conormal sheaf $\mathcal{I}_{X/Y}/\mathcal{I}_{X/Y}^{2}$ is a locally free sheaf of rank $c$ on $X$, where $\mathcal{I}_{X/Y}$ denotes the ideal sheaf of $X$ on $Y$.
Let $Y'\to Y$ be a morphism from an algebraic space $Y'$, and let $f'\colon X\times_{Y}Y'\hookrightarrow Y'$ be the base change of $f$.
Similarly to {\cite[Section~6.2]{Ful}}, the {\em Gysin pullback} $f^{!}\alpha$ by $f$ for a $d$-cycle $\alpha$ on $Y'$ is defined as follows:
We may assume that $\alpha=Z$ is an integral subspace of dimension $d$.
Then the normal cone
$$
C_{f'^{-1}(Z)/Z}:=\mathrm{Spec}_{f'^{-1}(Z)}\left( \bigoplus_{m\ge 0}\mathcal{I}_{f'^{-1}(Z)/Z}^{m}/\mathcal{I}_{f'^{-1}(Z)/Z}^{m+1}\right)
$$
is a closed subspace of the pullback of the normal bundle $N_{X/Y}=C_{X/Y}$ to $f'^{-1}(Z)$.
Then we define
$$
f^{!}\alpha:=s^{*}[C_{f'^{-1}(Z)/Z}],
$$
where $s^{*}$ is the Gysin pullback of the zero section 
$$
s^{*}\colon A_{d}(N_{X/Y}|_{f'^{-1}(Z)})\to A_{d-c}(f'^{-1}(Z)),
$$
that is, the inverse of the flat pullback $A_{d-c}(f'^{-1}(Z))\to A_{d}(N_{X/Y}|_{f'^{-1}(Z)})$ which is an isomorphism,
and we denote by $[C]$ the fundamental class of an algebraic space $C$.
Note that the cycle $f^{!}\alpha$ is determined up to rational equivalences on $f'^{-1}(\Supp \alpha)$.

A morphism $f\colon X\to Y$ between reduced and equidimensional algebraic spaces is called an {\em alteration}
if it is a dominant proper morphism and there exists an open dense subspace $U$ of $Y$ such that $f^{-1}(U)$ is dense in $X$ and the restriction map $f^{-1}(U)\to U$ is finite. 
If $X$ and $Y$ are integral schemes, then the definition of alterations coincides with the usual one.
An alteration $f\colon X\to Y$ is {\em of the same degree $d$}
if the restriction of $f$ over each irreducible component of $Y$ has degree $d$.

\subsection{Functorial divisors}
Let $\mathcal{X}$ be a reduced and equidimensional Artin stack of finite type over a field $k$. 
In this subsection, all schemes and algebraic spaces are assumed to be reduced, equidimensional, Zariski locally quasi-separated and of finite type over $k$, and
all divisors are considered with rational coefficients unless otherwise stated.

\begin{definition} \label{functorial}
A {\em functorial divisor} $D$ on $\mathcal{X}$ is a pair $D=(Z, \{D_{B}\}_{B})$ consisting of the following data:

\begin{enumerate}
\item
$Z$ is a nowhere dense closed subset $Z$ of $|\mathcal{X}|$ (automatically of codimension $1$), which is called the {\em support} of the functorial divisor $D$.

\item
For any morphism $\rho\colon B\to \mathcal{X}$ from an algebraic space $B$, 
a divisor class
$$
D_{B} \in A_{\dim B-1}(\rho^{-1}(Z))\otimes \Q
$$ is given, which is called the {\em trace of $D$ at $\rho$}, 
 and compatible with 
 alteration pushforwards, flat pullbacks and Gysin pullbacks.
 More precisely, the following three conditions hold:

\begin{enumerate}

\item
For any alteration morphism $\varphi\colon B'\to B$ over $\mathcal{X}$ of the same degree $d$,
$$
\varphi_{*}D_{B'}=dD_{B}.
$$

\item
For any flat morphism $\varphi\colon B'\to B$ over $\mathcal{X}$, 
$$
\varphi^{*}D_B=D_{B'}.
$$

\item
For any regular closed immersion $\varphi\colon B'\hookrightarrow B$ over $\mathcal{X}$,
$$
\varphi^{!}D_B=D_{B'}.
$$
\end{enumerate}

\end{enumerate}
Note that regular closed immersion in the condition (c) can be replaced by local complete intersection morphism.
A functorial divisor $D$ determines the usual Weil divisor $D_{\mathcal{X}}$ on $\mathcal{X}$
defined as the subfamily $\{D_{U}\}_{U}$ over the smooth morphisms $\rho\colon U\to \mathcal{X}$. 
Indeed, $\rho^{-1}(Z)\subseteq U$ is nowhere dense, and then it holds
$$
A_{\dim U-1}(\rho^{-1}(Z))=Z_{\dim U-1}(\rho^{-1}(Z)).
$$
By the same reason as above, $D_{B}$ is determined as a cycle on $B$ if any generic point of $B$ does not map into $Z$.
Note that the support of $D_{\mathcal{X}}$ is always contained in $Z$.

Two functorial divisors $D=(Z,\{D_{B}\}_{B})$ and $D'=(Z',\{D'_{B'}\}_{B})$ are {\em equivalent} if there exists a nowhere dense closed subset $Z''$ of $|\mathcal{X}|$ containing $Z$ and $Z'$ such that $D_{B}=D'_{B}$ holds in $A_{\dim B-1}(\rho^{-1}(Z''))\otimes \Q$ for each $\rho\colon B\to \mathcal{X}$.
The set of equivalence classes of functorial divisors on $\mathcal{X}$ forms an abelian group, which is denoted by $\mathrm{FDiv}_{\Q}(\mathcal{X})$.
There is a natural homomorphism 
$$
\mathrm{FDiv}_{\Q}(\mathcal{X})\to \WDiv(\mathcal{X})\otimes \Q;\quad D\mapsto D_{\mathcal{X}}.
$$

A functorial divisor $D=(Z, \{D_{B}\}_{B})$ on $\mathcal{X}$ is {\em effective} if $D_{B}$ is effective for any morphism $B\to \mathcal{X}$ which does not send any generic point of $B$ into $Z$.
\end{definition}

\begin{example} \label{functdivex}
(1) Cartier divisors on $\mathcal{X}$ are naturally regarded as functorial divisors:
Let $D$ be a Cartier divisor on $\mathcal{X}$.
For any morphism $\rho\colon B\to \mathcal{X}$ from an algebraic space $B$,
we consider the pullback $\rho^{*}D$ as a Cartier divisor and put $D_{B}:=(\rho^{*}D)^{\mathrm{W}}$ in $A_{\dim B-1}(\rho^{-1}(\Supp D))$.
Then $(\Supp D, \{D_{B}\}_{B})$ forms a functorial divisor.
This correspondence defines a natural homomorphism
$$
\CDiv(\mathcal{X})\otimes \Q \to \mathrm{FDiv}_{\Q}(\mathcal{X});\quad D\mapsto (\Supp D, \{D_{B}\}_{B}),
$$
which is injective if $\mathcal{X}$ is normal.

\smallskip

\noindent
(2) Let $\pi\colon \mathcal{Y}\to \mathcal{X}$ be a representable proper flat morphism of relative dimension $n$ between reduced and equidimensional Artin stacks.
Let $E_{1},\ldots,E_{n+1}$ be $n+1$ Cartier divisors on $\mathcal{Y}$ such that
the image $\pi(S)$ of the intersection $S:=\Supp E_{1}\cap \cdots \cap \Supp E_{n+1}$ of these supports is nowhere dense in $\mathcal{X}$.
Then we can define the intersection $E_{1}\cdot E_{2}\cdots E_{n+1}:=E_{1}\cdot E_{2}\cdots E_{n+1}\cdot [\mathcal{Y}]$ as a codimension $n+1$ cycle modulo rational equivalences on $S$, and obtain a Weil divisor 
$$
\pi_{*}(E_{1}\cdot E_{2}\cdots E_{n+1})
$$
 on $\mathcal{X}$ by taking proper pushforward (cf.\ \cite[(3.6)~Definition]{Vis}).
We can extend $\pi_{*}(E_{1}\cdot E_{2}\cdots E_{n+1})$ to a functorial divisor as follows:
Let $\rho\colon B\to \mathcal{X}$ be a morphism from an algebraic space $B$, and let 
$C$ be the algebraic space representing the fiber product $B\times_{\mathcal{X}}\mathcal{Y}$.
Let $\pi_{B}\colon C\to B$ and $\sigma\colon C\to \mathcal{Y}$ denote the natural projections.
We consider the pullbacks $E_{i,C}:=\sigma^{*}E_{i}$ as pseudo-divisors.
Then the intersection $E_{1, C}\cdot E_{2, C}\cdots E_{n+1, C}:=E_{1,C}\cdot E_{2,C}\cdots E_{n+1,C}\cdot [C]$ is defined as a cycle class in $A_{\dim C-n-1}(\sigma^{-1}(S))$.
Thus the proper pushforward $\pi_{B*}(E_{1, C}\cdot E_{2, C}\cdots E_{n+1, C})$ belongs to $A_{\dim B-1}(\rho^{-1}(\pi(S)))$.
It is easy to see that 
$$
(\rho^{-1}(\pi(S)), \{\pi_{B*}(E_{1, C}\cdot E_{2, C}\cdots E_{n+1, C})\}_{B})
$$
 forms a functorial divisor on $\mathcal{X}$ with the associated Weil divisor $\pi_{*}(E_{1}\cdot E_{2}\cdots E_{n+1})$.
The $n=0$ case is nothing but Example~\ref{functdivex}~(1).
\end{example}

The following is the main theorem in this section:
\begin{theorem} \label{effthm}
Assume that $\mathcal{X}$ is regular in codimension one.
Then, a functorial divisor $D$ on $\mathcal{X}$ is effective if and only if the associated Weil divisor $D_{\mathcal{X}}$ is effective.
\end{theorem}

\begin{proof}
Let $D=(Z, \{D_{B}\}_{B})$ be a functorial divisor on $\mathcal{X}$ such that $D_{\mathcal{X}}$ is effective.

\noindent
{\bf Step~1.} First we show that $D_{B}$ is ``numerically trivial'' over $\mathcal{X}$
for any morphism $\rho\colon B\to \mathcal{X}$ from a regular scheme $B$.
Let $C$ be a proper integral curve on $B$ such that the restriction map $\rho|_{C}$ factors through the spectrum of a field $k'$: 
$$
\rho|_{C}\colon C\to \Spec k' \to \mathcal{X}.
$$
Let $C'$ be a normalization of $C$ and denote by $\psi\colon C'\to B$ the natural morphism.
Since $C'$ and $B$ is regular, $\psi$ is local complete intersection.
Then $D_{C'}=\psi^{!}D_{B}$ holds by Definition~\ref{functorial}~(b), (c).
On the other hand, since the composition $\gamma\colon C'\to C\to \Spec k'$ is flat, we also have $D_{C'}=\gamma^{*}D_{\Spec k'}=0$. 
Thus, $D_{B}\cdot C=\deg \psi^{!}D_{B}=0$.

\noindent
{\bf Step~2.} In this step, we show that $D_{U'}$ is effective for any smooth morphism $U\to \mathcal{X}$ from a scheme $U$ and any alteration $U'\to U$ from a regular scheme $U'$. 
Let 
$$
U'\xrightarrow{\psi} \overline{U}\xrightarrow{\gamma} U
$$
 denote the Stein factorization of $U'\to U$, where $\overline{U}$ is normal, $\psi$ is proper birational and $\gamma$ is finite.
 By the negativity lemma (cf.\ {\cite[Lemma~3.39]{KoMo}}) and Step~1, it suffices to show that $D_{\overline{U}}$ is effective.
Let $U^{\circ}$ be the maximal open subset of $U$ over which $\gamma$ is flat, and put $\overline{U}^{\circ}:=\gamma^{-1}(U^{\circ})$.
Note that $\gamma$ is flat in codimension $1$ since $\overline{U}$ is normal and $U$ is regular in codimension $1$ (cf.\ {\cite[Corollary~5.5]{KoMo}}).
Then $D_{\overline{U}^{\circ}}=\gamma^{*}D_{U^{\circ}}$ holds by Definition~\ref{functorial}~(b).
Since $D_{U^{\circ}}$ is effective by assumption, so is $D_{\overline{U}^{\circ}}$.
Since the complement of $\overline{U}^{\circ}$ in $\overline{U}$ has codimension greater than $1$, it follows that $D_{\overline{U}}$ is effective.

\noindent
{\bf Step~3.}
We show that $D_{B}$ is effective for any morphism $\rho\colon B\to \mathcal{X}$ from an algebraic space $B$ which does not send any generic point of $B$ into $Z$.
Since the effectivity can be checked smooth locally, we may assume that
$B$ is an affine scheme and there exists a smooth morphism $U\to \mathcal{X}$ from an affine scheme $U$ such that $\rho$ factors through $U$.
Moreover, we may assume that $B$ is irreducible since the effectivity can be checked 
on each irreducible component of the normalization of $B$.
We take an alteration $U'\to U$ from a regular scheme $U'$
and an alteration $\varphi\colon B'\to B$ from a regular integral scheme $B'$ such that
$B'\to B\to U$ factors through $U'$.
Note that such alterations always exist (see \cite{dJo}).
Then the factored morphism $\psi\colon B'\to U'$ is local complete intersection since $B'$ and $U'$ are regular.
By Step~2, $D_{U'}$ is effective.
By the definition of functorial divisors, we have
$$
 D_{B}=\frac{1}{\deg \varphi} \varphi_{*}D_{B'}=\frac{1}{\deg \varphi} \varphi_{*}\psi^{!}D_{U'}
$$
and hence it is effective.
\end{proof}

\section{Moduli of curves}
\label{sec:Moduli of curves}

\subsection{Curves}
In this subsection, we first define various properties of curves used in this paper (some of the terminology used here are non-standard).

\begin{definition} \label{curvedef}
Let $k$ be an algebraically closed field.
In this paper, a {\em curve} $C$ over $k$ means a proper $1$-dimensional scheme over $k$.
A Gorenstein curve $C$ over $k$ is {\em canonically polarized} (resp.\ {\em canonically quasi-polarized}) if the dualizing sheaf $\omega_C$ is ample (resp.\ nef, that is, $\deg \omega_{C}|_{C_i}\ge 0$ for each integral subcurve $C_{i}$ of $C$), {\em of genus $g$} if $h^{0}(\O_C)=1$ and $h^{0}(\omega_C)=g$, and {\em smoothable} if there exists a smoothing family $f\colon \mathcal{C}\to \Spec{R}$ (that is, a proper flat morphism with smooth generic fiber) over a divisorial valuation ring $R$ over $k$ such that the central fiber is isomorphic to $C$.
Furthermore, when $\mathcal{C}$ can be taken to have at most Du Val singularities (resp.\ to be regular), such $f$ is called a {\em good smoothing family} (resp.\ {\em very good smoothing family}) and $C$ is called {\em good smoothable} (resp.\ {\em very good smoothable}).
Note that when $C$ is reduced, then it is very good smoothable if and only if it has only plane curve singularities,
since any plane curve singularity has a one-parameter smoothing with a smooth total space and there are no local-to-global obstructions to deformations of the reduced curve ({\cite[Tag 0DZP]{Sta}}).

A Gorenstein curve $C$ of genus $g$ is {\em of minimal fiber type} (resp.\ {\em of canonical fiber type}) if it is Gorenstein, canonically quasi-polarized and very good smoothable (resp.\ Gorenstein, canonically polarized and good smoothable). 
Any fiber of relatively minimal (resp.\ relative canonical models of) fibered surfaces is of minimal fiber type (resp.\ of canonical fiber type).
\end{definition}

\begin{remark} \label{curvepropertyrem}
It is clear that good smoothable implies local complete intersection and that local complete intersection implies Gorenstein.
Any reduced and local complete intersection curve is smoothable ({\cite[Tag 0E7Y]{Sta}}).
\end{remark}

The following lemma is a generalization of {\cite[Theorem~(1.2)]{DeMu}} and {\cite[Proposition~4.3]{HaHy}} to canonically polarized Gorenstein curves:

\begin{lemma}\label{propertyofcurves}
Let $C$ be a Gorenstein canonically polarized curve of genus $g\ge 2$ over an algebraically closed field $k$.
Then the following hold.

\smallskip

\noindent
$\mathrm{(1)}$ 
$h^0(\omega_{C}^{\otimes n})=(2n-1)(g-1)$ and $h^{1}(\omega_{C}^{\otimes n})=0$ for $n\ge 2$.

\smallskip

\noindent
$\mathrm{(2)}$ 
$\omega_{C}^{\otimes n}$ is basepoint-free for $n\ge 2g$, and very ample for $n\ge 2g+1$.

\smallskip

\noindent
$\mathrm{(3)}$
If $C$ is good smoothable, then $\omega_{C}^{\otimes n}$ is basepoint-free for $n\ge 3$, and very ample for $n\ge 4$.
If, in addition, $C$ is not a scheme-theoretic multiple of multiplicity $2g-2$ of its reduction, then the same holds for $n\ge 2$ and $n\ge 3$, respectively.

\smallskip

\noindent
$\mathrm{(4)}$
If $C$ is very good smoothable, then $\omega_{C}^{\otimes n}$ is basepoint-free for $n\ge 2$, and very ample for $n\ge 3$.
If, in addition, $C$ is not a scheme-theoretic multiple of multiplicity $g-1$ of its reduction or has no elliptic tails, then $\omega_{C}^{\otimes 2}$ is very ample.
\end{lemma}

\begin{proof}
(1): From the Serre duality and the ampleness of $\omega_{C}$, we have 
$$
h^{1}(\omega_{C}^{\otimes n})=h^{0}(\omega_{C}^{\otimes 1-n})=0.
$$
It follows from the the Riemann--Roch theorem that
$$
h^{0}(\omega_{C}^{\otimes n})=\chi(\omega_{C}^{\otimes n})=\chi(\O_{C})+n\deg{\omega_{C}}=(2n-1)(g-1),
$$
where we use $\chi(\O_C)=1-g$ and $\deg{\omega_C}=2g-2$.

(2): 
The claim directly follows from {\cite[Theorem~1.1]{CFHR}}.
Indeed, the assumption of {\cite[Theorem~1.1]{CFHR}} is satisfied for $H=\omega_{C}^{\otimes n}$ since for any subcurve $B\subseteq C$, we have $p_a(B)=1-\chi(\O_{B})\le h^{1}(\O_{B})\le h^{1}(\O_{C})=g$ and $H\cdot B\ge n$.

(3), (4): 
Let $f\colon \mathcal{C}\to \Spec R$ be a good smoothing family of $C$, that is, a proper flat morphism with a smooth generic fiber over a divisorial valuation ring $R$ over $k$ such that the central fiber is isomorphic to $C$ and $\mathcal{C}$ has at most Du Val singularities.
Let $\omega_{f}$ denote the relative dualizing sheaf of $f$, and let $K_f$ be a canonical divisor on $\mathcal{C}$ with $\omega_{f}\cong \O_{\mathcal{C}}(K_f)$.
From (1) and the cohomology and base change theorem, to prove the claim for basepoint-freeness (resp.\ very-ampleness), it suffices to show that $\omega_{f}^{\otimes n}$ is $f$-free (resp.\ $f$-very ample).
To prove this, we use a Reider-type theorem established in \cite{Eno}:
Let $D$ be an $f$-nef and $f$-big Cartier divisor on $\mathcal{C}$ and $\zeta$ a $0$-dimensional subscheme of $\mathcal{C}$ (automatically contained in the central fiber $C$).
According to {\cite[Theorem~5.2]{Eno}}, there exists a positive number $\delta_{\zeta}$ depending only on the germ $\zeta\subseteq \mathcal{C}$ such that if $(D-B)\cdot B>\delta_{\zeta}/4$ holds for any $f$-vertical effective $1$-cycle $B$ on $\mathcal{C}$ intersecting $\zeta$, then the restriction map
$$
f_{*}\O_{\mathcal{C}}(K_f+D)\to f_{*}\O_{\mathcal{C}}(K_f+D)|_{\zeta}
$$
is surjective.
To prove the basepoint-freeness, it suffices to verify the assumption of this claim when $D=(n-1)K_f$ and $\zeta$ has length $1$.
In this situation, we may assume $\delta_{\zeta}\le 4$ using {\cite[Lemma~5.9]{Eno}}.
Let $B$ be any $f$-vertical effective $1$-cycle on $\mathcal{C}$ that intersects $\zeta$.
If $n\ge 3$, we have
$$
((n-1)K_f-B)\cdot B\ge n-1-B^2 >1\ge \frac{\delta_{\zeta}}{4},
$$
where the first inequality comes from the fact that $K_f$ is Cartier and $f$-ample
and the second inequality comes from $n\ge 3$ and $B^2\le 0$ due to the negative semi-definiteness of the intersection form on the fiber $C$.
Therefore, $\omega_{f}^{\otimes n}$ is $f$-free for $n\ge 3$.

For the very ample case, we consider $D=(n-1)K_f$ and either $\zeta$ is of length $2$ whose support contains a non-singular point of $C$ or $\zeta$ is defined by the square of the maximal ideal $\mathfrak{m}_{x}^{2}$ at a Du Val singularity $x$ in $C$.
Thus we may assume that $\delta_{\zeta}\le 8$ ({\cite[Remark~5.10]{Eno}}).
If $n\ge 4$, we have
$$
((n-1)K_f-B)\cdot B\ge n-1-B^2 >2\ge \frac{\delta_{\zeta}}{4},
$$
whence $\omega_{f}^{\otimes n}$ is $f$-very ample for $n\ge 4$.

If $\omega_{f}^{\otimes 2}$ is not $f$-free or $\omega_{f}^{\otimes 3}$ is not $f$-very ample, then we obtain $K_f\cdot B=1$ and $B^2=0$ from the above argument.
It follows that $C=mB$ for some $m\ge 1$.
Since $K_f\cdot C=2g-2$ and $K_{f}\cdot B=1$, we have $m=2g-2$ and $B$ is irreducible.
Note that if $\mathcal{C}$ is smooth, the above case does not occur since $K_{f}\cdot B+B^{2}=2p_a(B)-2$ is even.

If $\mathcal{C}$ is smooth and $\omega_{f}^{\otimes 2}$ is not $f$-very ample, then $K_{f}\cdot B-B^2=2$ from the above argument.
Thus, either $K_{f}\cdot B=2$ and $B^{2}=0$ or $K_{f}\cdot B=1$ and $B^{2}=-1$.
The former case implies that $C=(g-1)B$ with $p_a(B)=2$.
Note that $B$ is reduced and has at most two irreducible components.
The latter case implies that $B$ is an integral curve with $p_a(B)=1$ and $(C-B)\cdot B=1$, that is, $B$ is an elliptic tail.
\end{proof}

\subsection{Families and moduli of curves}
\begin{definition} \label{*curvedef}
Let $*$ be a property of curves (e.g., $*=$ Gorenstein, canonically polarized and smoothable of genus $g$).
A morphism of algebraic spaces $f\colon C\to B$ is called a {\em family of $*$-curves} if $f$ is a proper flat morphism of finite presentation and any geometric fiber of $f$ is a curve satisfying $*$.
Let $\mathcal{M}^{*}$ be the category fibered in groupoids over the category of schemes
whose objects are families of $*$-curves, which is called the {\em moduli stack of $*$-curves}.
It is also denoted by $\mathcal{M}^{*}_{g}$ if all $*$-curves are of genus $g$.

If the condition $*$ is empty ($*=\emptyset$), then the moduli stack $\mathcal{M}^{\emptyset}$ is an Artin stack (\cite{Hal}, {\cite[Tag 0D5A]{Sta}}).
If the property $*$ is an open condition for curves, that is, 
for any family of curves $f\colon C\to B$, the set
of points $p\in B$ such that the geometric fiber $f^{-1}(\overline{p})$ satisfies $*$ forms an open subset of $B$, then $\mathcal{M}^{*}$ is an Artin stack since it is an open substack of $\mathcal{M}^{\emptyset}$.
The following properties of curves are open conditions:
Cohen--Macaulay, Gorenstein, local complete intersection, stable, smooth, reduced, smoothable, canonically polarized, very ample $n$-th canonical system, of genus $g$.
\end{definition}

The following lemma ensures that base changes of families of Gorenstein canonically polarized curves behave well:

\begin{lemma} \label{relativedualizing}
Let $f\colon C\to B$ be a family of Gorenstein canonically polarized curves of genus $g$.
Then the following hold:

\smallskip

\noindent
$(1)$ $f_{*}\omega_{f}^{\otimes n}$ is locally free and compatible with base change for any $n\ge 1$.

\smallskip

\noindent
$(2)$ $R^{1}f_{*}\omega_{f}^{\otimes n}=0$ for $n\ge 2$ and $R^{1}f_{*}\omega_{f}\cong \O_B$.

\smallskip

\noindent
$(3)$ If any geometric fiber of $f$ has free $($resp.\ very ample$)$ $n$-th canonical linear system, then $\omega_{f}^{\otimes n}$ is $f$-free $($resp.\ $f$-very ample$)$.
\end{lemma}

\begin{proof}
Note that the relative dualizing sheaf $\omega_{f}$ exists as an invertible sheaf on $C$ and compatible with base change since $f$ is Gorenstein.
We may assume that $B$ is a scheme.

The claims (1) and (2) for $n\ge 2$ follow from Lemma~\ref{propertyofcurves}~(1) and the cohomology and base change theorem (\cite[III, Theorem~12.11]{Har}.
Note that it holds without Noetherian assumption for proper morphisms of finite presentation by a standard limit argument. cf.\ \cite[Lemma~5.1.1]{Con}).
The claim (1) for $n=1$ also follows from the claim (2) for $n=1$ by the cohomology and base change theorem.
Now we show that the trace map 
$$
\mathrm{Tr}_{f}\colon R^{1}f_{*}\omega_{f}\to \O_B
$$
 (cf.\ \cite[Corollary~3.6.6]{Con}) is an isomorphism.
First note that the trace map $H^{1}(\omega_{f^{-1}(p)})\to k(p)$
on any fiber $f^{-1}(p)$ of $f$ is an isomorphism since $H^{0}(\O_{f^{-1}(p)})\cong k(p)$ (due to the definition of genus $g$) and the Serre duality.
Since the trace map is compatible with base change and $R^{2}f_{*}\omega_{f}=0$, it follows from Nakayama's lemma and the cohomology and base change theorem that $\mathrm{Tr}_{f}$ is surjective and $R^{1}f_{*}\omega_{f}$ is locally generated by one element.
Since any local generator of $R^{1}f_{*}\omega_{f}$ sends to a unit of $\O_B$, it follows that $\mathrm{Tr}_{f}$ is injective and so an isomorphism.

The claim (3) can easily seen by standard arguments using (1) and the cohomology and base change theorem.
\end{proof}

In this paper, we mainly use the following five moduli stacks of Gorenstein canonically polarized curves
$$
\mathcal{M}_{g}\subseteq \overline{\mathcal{M}}_{g} \subseteq \mathcal{M}^{\mathrm{lci+}}_{g} \subseteq \mathcal{M}^{\mathrm{lci}}_{g} \subseteq \mathcal{M}^{\mathrm{Gor}}_{g},
$$
which are defined as follows:

\begin{example} \label{*ex}
(1) Consider the condition
$$
*=\text{smooth projective curves of genus $g$ (resp.\ stable curves of genus $g$),}
$$
Then, the corresponding moduli stack $\mathcal{M}^{*}_{g}$ is usually denoted by $\mathcal{M}_{g}$ (resp.\ $\overline{\mathcal{M}}_{g}$).
It is well known that $\overline{\mathcal{M}}_{g}$ is a smooth Deligne--Mumford stack of relative dimension $3g-3$ over $\Z$ with projective course moduli space and $\mathcal{M}_{g}$ is an open dense substack of $\overline{\mathcal{M}}_{g}$ (\cite{DeMu}).

\smallskip

\noindent
(2) 
Consider the condition
$$
*=\text{Gorenstein (resp.\ local complete intersection), canonically polarized and of genus $g$,}
$$
Then, the corresponding moduli stack $\mathcal{M}^{*}_{g}$ is an Artin stack of finite type over $\Z$.
Indeed, similarly to $\overline{\mathcal{M}}_{g}$, Lemmas~\ref{propertyofcurves} and \ref{relativedualizing} ensure that this stack can be realized as a quotient stack of a locally closed subscheme of a suitable Hilbert scheme by the $\mathrm{PGL}$-action.
Let $\mathcal{M}^{\mathrm{Gor}}_{g}$ (resp.\ $\mathcal{M}^{\mathrm{lci}}_{g}$) be the normalization of the irreducible component of $\mathcal{M}^{*}_{g}$ containing $\mathcal{M}_{g}$.
By the universality of the normalization, for a normal scheme $B$, the groupoid $\mathcal{M}^{\mathrm{Gor}}_{g}(B)$ (resp.\ $\mathcal{M}^{\mathrm{lci}}_{g}(B)$) consists of all families of smoothable Gorenstein (resp.\ local complete intersection) canonically polarized curves of genus $g$ over $B$.
Note that for any fibered surface $f\colon S\to B$ of genus $g$, the relative canonical model $\overline{f}\colon \overline{S}\to B$ belongs to $\mathcal{M}^{\mathrm{lci}}_{g}$.
Note also that all moduli stacks currently known to appear in the Hassett--Keel program (\cite{Has}, \cite{HyLe}, \cite{HaHy}, \cite{AFSW}, \cite{Zha}, \cite{ADLW}) are contained in $\mathcal{M}^{\mathrm{lci}}_{g}$.

\smallskip
 
\noindent
(3)
Consider the condition
$$
*=\text{reduced,  local complete intersection, canonically polarized and of genus $g$.}
$$
Then, $\mathcal{M}^{\mathrm{lci+}}_{g}:=\mathcal{M}^{*}_{g}$ is a smooth Artin stack of finite type over $\Z$ containing $\overline{\mathcal{M}}_{g}$ as an open dense substack.
Indeed, the unobstructedness of deformations of reduced local complete intersection curves ({\cite[Tag 0DZX]{Sta}}) implies the smoothness of $\mathcal{M}^{\mathrm{lci+}}_{g}$.
We will show later that the codimension of the complement $\mathcal{M}^{\mathrm{lci+}}_{g}\setminus \overline{\mathcal{M}}_{g}$ is greater than one (Theorem~\ref{codimbound}). 
\end{example}

\begin{theorem} \label{codimbound}
The complement of $\overline{\mathcal{M}}_{g}$ in $\mathcal{M}^{\mathrm{lci+}}_{g}$ has no codimension one components.
\end{theorem}

\begin{proof}
By discussing each geometric fiber of $\mathcal{M}^{\mathrm{lci+}}_{g}\to \Spec \Z$, we may assume that $\mathcal{M}^{\mathrm{lci+}}_{g}$ is defined over an algebraically closed field.
Let $\Delta\subseteq \mathcal{M}^{\mathrm{lci+}}_{g}$ be a codimension one boundary component.
It suffices to show that $\Delta$ is the closure of $\delta_{i}$ for some $i$.
Let $[C]\in \Delta$ be a general point corresponding to a curve $C$.
Then, the tangent space of $\mathcal{M}^{\mathrm{lci+}}_{g}$ at $[C]$ can be regarded as the tangent space of the deformation functor of $C$.
Since $[C]\in \Delta$ is general, we may assume that any deformation of $C$ is a smoothing or an equisingular deformation along $\Delta$.
Since the local-to-global obstruction to deformations of $C$ vanishes ({\cite[Tag 0DZP]{Sta}}) and $\Delta$ is of codimension one, $C$ has a unique singularity $p\in C$ and the germ $[C]\in \Delta$ corresponds to equisingular deformations of the singularity $p$ that forms in codimension one in the deformation space of $p\in C$.
If $p\in C$ is not planer, then there exists a deformation to a singularity whose embedding dimension decreases by exactly one, obtained by perturbing one of the defining equations, which is a contradiction.
Thus, $p\in C$ is a plane curve singularity.
Moreover, perturbing the defining equation of $p\in C$, we can take a deformation to a node.
Thus, $p$ is already a node.
\end{proof}

\subsection{Curves of minimal fiber type}
In this subsection, we consider the base change property of families of curves of minimal fiber type.

\begin{lemma} \label{curvecanomodel}
Let $C$ be a Gorenstein curve of genus $g\ge 2$ of minimal fiber type over an algebraically closed field $k$.
Then, the canonical model 
$$
\overline{C}:=\Proj \oplus_{n\ge 0} H^{0}(\omega_{C}^{\otimes n})
$$
 is Gorentein of genus $g$ and of canonical fiber type satisfying $H^{i}(\omega_{C}^{\otimes n})\cong H^{i}(\omega_{\overline{C}}^{\otimes n})$ for each $i\ge 0$ and $n\ge 0$, and the natural morphism $C\to \overline{C}$ contracts all smooth rational curves on which $\omega_{C}$ is trivial.
In particular, $H^{i}(\omega_{C}^{\otimes n})=0$ for $i\ge 1$ and $n\ge 2$, or $i\ge 2$.
\end{lemma}

\begin{proof}
We take a very good smoothing family $f\colon \mathcal{C}\to \Spec R$ of $C$.
Let 
$$
\overline{f}\colon \overline{\mathcal{C}}:=\Proj \oplus_{n\ge 0}f_{*}\omega_{f}^{\otimes n}\to \Spec R
$$
 denote the relative canonical model of $f$.
We show that the central fiber $\overline{\mathcal{C}}_{0}=\overline{f}^{-1}(0)$ coincides with the canonical model of $C$.
Note that $\overline{\mathcal{C}}$ has at most Du Val singularities and the natural morphism $\pi\colon \mathcal{C}\to \overline{\mathcal{C}}$ is the minimal resolution.
Thus we have $R^{i}f_{*}\omega_{f}^{\otimes n}\cong R^{i}\overline{f}_{*}\omega_{\overline{f}}^{\otimes n}$ for any $n\ge 0$ and $i\ge 0$.
Since $R^{i}\overline{f}_{*}\omega_{\overline{f}}^{\otimes n}$ is locally free by Lemma~\ref{relativedualizing}, 
we obtain 
$$
H^{i}(\omega_{\overline{\mathcal{C}}_{0}}^{\otimes n})\cong R^{i}\overline{f}_{*}\omega_{\overline{f}}^{\otimes n}\otimes \kappa(0) \cong R^{i}f_{*}\omega_{f}^{\otimes n}\otimes \kappa(0)\cong H^{i}(\omega_{C}^{\otimes n})
$$
by the cohomology and base change theorem, whence the claim holds.
\end{proof}

Lemma~\ref{propertyofcurves} and Lemma~\ref{curvecanomodel} immediately imply the following:

\begin{corollary}
Let $C$ be a Gorenstein curve of genus $g\ge 2$ of minimal fiber type over an algebraically closed field $k$.
Then the following hold.

\smallskip

\noindent
$\mathrm{(1)}$ $h^0(\omega_{C}^{\otimes n})=(2n-1)(g-1)$ and $h^{1}(\omega_{C}^{\otimes n})=0$ for $n\ge 2$.

\smallskip

\noindent
$\mathrm{(2)}$ The $n$-th canonical linear system $|\omega_{C}^{\otimes n}|$ is basepoint-free for $n\ge 3$ and defines a birational map onto image for $n\ge 4$.

\smallskip

\noindent
$\mathrm{(3)}$ If the canonical model of $C$ is not a scheme-theoretic multiple of multiplicity $2g-2$ of its reduction, then $|\omega_{C}^{\otimes n}|$ is basepoint-free for $n\ge 2$ and defines a birational map onto image for $n\ge 3$.
\end{corollary}

\begin{lemma} \label{relcanocomp}
Let $f\colon C\to B$ be a family of Gorenstein curves of genus $g\ge 2$ of minimal fiber type over an algebraic space $B$.
Let 
$$
\overline{f}\colon \overline{C}:=\Proj_{B}\oplus_{n\ge 0}f_{*}\omega_{f}^{\otimes n}\to B
$$
 denote the relative canonical model of $f$.
Then, $\overline{f}$ is a family of Gorenstein curves of genus $g$ of canonical fiber type and the base change $\overline{C}\times_{B}B'\to B'$ of $\overline{f}$ for any morphism $B'\to B$ coincides with the relative canonical model of $C\times_{B}B'\to B'$.
\end{lemma}

\begin{proof}
We may assume that $B$ is a scheme.
By the cohomology and base change theorem and Lemma~\ref{curvecanomodel}, it suffices to show that
$$
f_{*}\omega_{f}^{\otimes n}\otimes \kappa(p)\cong H^{0}(\omega_{f^{-1}(p)}^{\otimes n})
$$
 for any $n\ge 0$ and any point $p\in B$.
By using the cohomology and base change theorem repeatedly, we have 
$$
R^{i}f_{*}\omega_{f}^{\otimes n}\otimes \kappa(p) \cong H^{i}(\omega_{f^{-1}(p)}^{\otimes n})=0
$$
 for $i\ge 2$.
Hence $R^{i}f_{*}\omega_{f}^{\otimes n}=0$ for $i\ge 2$ and $R^{1}f_{*}\omega_{f}^{\otimes n}\otimes \kappa(p) \cong H^{1}(\omega_{f^{-1}(p)}^{\otimes n})$.
In the case that $n\ge 2$, it follows from Lemma~\ref{curvecanomodel} that $H^{1}(\omega_{f^{-1}(p)}^{\otimes n})=0$.
Thus $R^{1}f_{*}\omega_{f}^{\otimes n}=0$ and hence
we have $f_{*}\omega_{f}^{\otimes n}\otimes \kappa(p) \cong H^{0}(\omega_{f^{-1}(p)}^{\otimes n})$.
For $n=1$, then the trace map $R^{1}f_{*}\omega_{f}\to \O_{B}$ is an isomorphism as in the proof of Lemma~\ref{relativedualizing}.
Thus we have $f_{*}\omega_{f}\otimes \kappa(p) \cong H^{0}(\omega_{f^{-1}(p)})$ by using the cohomology and base change theorem again.
\end{proof}

\section{Morsification conjecture}
\label{sec:Morsification conjecture}

In this section, we work over an algebraically closed field $k$ 
and consider relatively minimal fibrations $f\colon S\to B$ of genus $g$ over a curve $B$ that is not necessarily complete, that is, $f$ is a proper surjective morphism from a smooth surface $S$ to a smooth curve $B$ whose general fiber is a smooth projective curve of genus $g$ such that $K_{S}$ is relatively nef over $B$.
We denote by $f\colon S\to (p\in B)$ or simply $f^{-1}(p)$ the fiber germ of $f$ at a closed point $p$ of $B$.
In this paper, all fiber germs are considered \'{e}tale locally on the base $p\in B$.
Note that it is also possible to treat fiber germs in the category of complex analytic spaces.
In this case, it is necessary to replace the moduli stacks of curves by their analytifications.

\subsection{Morsification conjecture}

\begin{definition}[Splitting deformation/family, atomic fibers] \label{splittingdef}
Let $f\colon S\to (p\in B)$ be a relatively minimal fiber germ of genus $g$.
We consider a one-parameter family 
\begin{equation} \nonumber 
\mathcal{S}\xrightarrow{\bm{f}} (p\in \mathcal{B}) \xrightarrow{h} (0\in T)
\end{equation}
 of relatively minimal fibrations with the central fiber germ $\bm{f}_{0}\colon \mathcal{S}_{0}\to (p\in \mathcal{B}_{0})$ at $0\in T$ is isomorphic to $f\colon S\to (p\in B)$,
where $\bm{f}\colon \mathcal{S}\to \mathcal{B}$ is a family of curves over a smooth surface $\mathcal{B}$ and $h\colon \mathcal{B}\to T$ is a smooth morphism to a smooth curve $T$ with $h(p)=0$. 
We call this a {\em splitting family} of $f^{-1}(p)$.
The base change $\bm{f}_{t}\colon \mathcal{S}_{t}\to \mathcal{B}_{t}$ of $\bm{f}$ at $t\in T$ is called a {\em splitting deformation} of $f^{-1}(p)$.
Note that this is not a fiber germ and may have many singular fibers.
A splitting family is called {\em proper} if $\bm{f}_{t}\colon \mathcal{S}_{t}\to \mathcal{B}_{t}$ has at least two singular fibers for each $t\neq 0$ sufficiently near $0$.
We say that a fiber germ $f^{-1}(p)$ is {\em atomic} if there is no proper splitting family of $f^{-1}(p)$.
\end{definition}

The following is a famous conjecture due to Xiao and Reid (cf.\ \cite{Rei}):

\begin{conjecture}[Morsification conjecture] \label{Morsconj}
In characteristic $0$ or the complex analytic setting,
a fiber germ $f^{-1}(p)$ is atomic if and only if it is a stable curve with one node or a multiple of a smooth curve.
\end{conjecture}

\begin{remark}
(1) The if part in Conjecture~\ref{Morsconj} is known ({\cite[Theorem~2.0.2]{Tak}}) and the only if part is known for $g\le 6$ in the complex analytic setting (\cite{Hor3}, \cite{TakIII}, \cite{Oku}).

\smallskip

\noindent
(2) If an algebraic fiber germ $f^{-1}(p)$ defined over $\mathbb{C}$ is atomic in the complex analytic setting, then it is atomic in the algebraic setting.
By this and the Lefschetz principle, the if part in Conjecture~\ref{Morsconj} holds true over any algebraically closed field of characteristic $0$.
\end{remark}

\begin{definition} \label{splittable}
We consider a set $\mathcal{A}$ of isomorphism classes of relatively minimal fiber germs of genus $g$ including all smooth fiber germs of genus $g$.
Then, a relatively minimal fiber germ $f\colon S\to (p\in B)$ is {\em splittable into fibers in $\mathcal{A}$} if there exists a splitting family $\mathcal{S}\xrightarrow{\bm{f}} (p\in \mathcal{B})\xrightarrow{h} (0\in T)$ of $f$ such that for any non-zero $t\in T$ sufficiently near $0$ and any closed point $q\in \mathcal{B}_{t}$, the fiber germ $\bm{f}_{t}\colon \mathcal{S}_{t}\to (q\in \mathcal{B}_{t})$ belongs to $\mathcal{A}$.
Let $\mathcal{A}_{1}$ denote the set of isomorphism classes of relatively minimal fiber germs of genus $g$ which is splittable into fibers in $\mathcal{A}$.
Clearly we have $\mathcal{A}\subseteq \mathcal{A}_{1}$.
For a positive integer $n$, we inductively define 
$$
\mathcal{A}_{0}:=\mathcal{A}, \quad
\mathcal{A}_{n}:=(\mathcal{A}_{n-1})_{1}, \quad
\mathcal{A}_{\infty}:=\cup_{n\ge 0}\mathcal{A}_{n}.
$$
A fiber germ $f^{-1}(p)$ is said to be {\em splittable into fibers in $\mathcal{A}$ by finitely many splitting deformations} if $f^{-1}(p)$ belongs to $\mathcal{A}_{\infty}$.

Using these notions, Conjecture~\ref{Morsconj} is equivalent to the following:
Let $\mathcal{A}^{\mathrm{atm}}$ be the set of isomorphism classes of stable fiber germs with at most one node and smooth multiple fiber germs.
Then $\mathcal{A}^{\mathrm{atm}}_{\infty}$ consists of all isomorphism classes of relatively minimal fiber germs.

A relatively minimal fiber germ $f\colon S\to (p\in B)$ {\em satisfies the Morsification conjecture} if it belongs to $\mathcal{A}^{\mathrm{atm}}_{\infty}$.
\end{definition}

Let $\overline{f}\colon \overline{S}\to (p\in B)$ be the relative canonical model of a fiber germ $f\colon S\to (p\in B)$.
We can also consider the splitting family 
$$
\overline{\mathcal{S}}\xrightarrow{\overline{\bm{f}}} (p\in \mathcal{B}) \xrightarrow{h} (0\in T)
$$
 of $\overline{f}$ by the same definition of Definition~\ref{splittingdef}. 
Note that the fiber $\overline{\mathcal{S}}_{t}$ has at most Du Val singularities for any $t\in T$ sufficiently near $0$.

\begin{lemma} \label{twosplitfam}
Let $f\colon S\to (p\in B)$ be a relatively minimal fiber germ of genus $g\ge 2$, and let $\overline{f}\colon \overline{S} \to (p\in B)$ be the relative canonical model of $f$.

\smallskip

\noindent
$(1)$ Let $\mathcal{S}\xrightarrow{\bm{f}} (p\in \mathcal{B})\xrightarrow{h} (0\in T)$ be a splitting family of $f^{-1}(p)$.
Let $\overline{\bm{f}}\colon \overline{\mathcal{S}}\to (p\in \mathcal{B})$ denote the relative canonical model of $\bm{f}$.
Then, the composition 
$$
\overline{\mathcal{S}}\xrightarrow{\overline{\bm{f}}} (p\in \mathcal{B}) \xrightarrow{h} (0\in T)
$$ 
is a splitting family of $\overline{f}^{-1}(p)$.

\smallskip

\noindent
$(2)$ Let $\overline{\mathcal{S}}\xrightarrow{\overline{\bm{f}}} (p\in \mathcal{B}) \xrightarrow{h} (0\in T)$ be a splitting family of $\overline{f}^{-1}(p)$.
Then, there exist a finite morphism $(0'\in T')\to (0\in T)$  and a simultaneous resolution $\pi\colon \mathcal{S'}\to \overline{\mathcal{S'}}:=\overline{\mathcal{S}}\times_{T}T'$ over $T'$ such that the composition 
$$
\mathcal{S'}\xrightarrow{\overline{\bm{f'}} \circ \pi} (p' \in \mathcal{B'}) \xrightarrow{h'} (0'\in T')
$$
 of $\pi$ and the base change of $h\circ \overline{\bm{f}}$ is a splitting family of $f^{-1}(p)$.
\end{lemma}

\begin{proof}
The claim (1) follows from Lemma~\ref{relcanocomp}.
The claim (2) follows from the existence of simultaneous minimal resolutions of Du Val singularities (cf.\ {\cite[Theorem~4.28]{KoMo}}).
\end{proof}

\subsection{Moduli conjecture}
We consider relations between the moduli stacks of curves and splitting deformations.

\begin{proposition} \label{splitcri}
Any relatively minimal fiber germ whose canonical model has reduced fibers is splittable into stable fibers with at most one node.
In particular, it satisfies the Morsification conjecture.
\end{proposition}

\begin{proof}
Let $f\colon S\to (p\in B)$ be a relatively minimal fiber germ of genus $g$ with the relative canonical model $\overline{f}\colon \overline{S}\to (p\in B)$.
Assume that $\overline{f}^{-1}(p)$ is reduced.
We consider the moduli stack $\mathcal{M}_{g}^{\mathrm{lci+}}$ introduced in Example~\ref{*ex}~(3) over the base field $k$.
Thus, $\overline{f}$ belongs to $\mathcal{M}_{g}^{\mathrm{lci+}}$ and
we denote by $\rho_{\overline{f}}\colon B\to \mathcal{M}_{g}^{\mathrm{lci+}}$ the moduli map of $\overline{f}$.
We take a smooth morphism $U\to \mathcal{M}_{g}^{\mathrm{lci+}}$ from an affine scheme $U$ whose image contains the moduli point of $\overline{f}^{-1}(p)$.
Since $\mathcal{M}_{g}^{\mathrm{lci+}}$ is smooth over $k$ ({\cite[Tag 0DZX]{Sta}}), so is $U$.
Replacing $(p\in B)$ with its \'{e}tale neighborhood, we may assume that the moduli map of $\overline{f}$ factors through $U$, which is denoted by 
$$
\rho_{\overline{f}}\colon B\xrightarrow{\tau} U\to \mathcal{M}_{g}^{\mathrm{lci+}}.
$$
Taking the graph $B\to B\times U$ of $\tau$ and replacing $U$ with $B\times U$, we may assume that $\tau$ is a regular closed immersion.
Take a smooth surface $\mathcal{B}$ defined as the complete intersection of general hypersurfaces in $U$ containing $B$.
By Theorem~\ref{codimbound}, we may assume that the image of $\mathcal{B}\setminus \{p\}$ to $\mathcal{M}_{g}^{\mathrm{lci+}}$ is contained in $\overline{\mathcal{M}}_{g}$.
Furthermore, taking hyperplanes defining $\mathcal{B}$ generally,
we may assume that the surface $\mathcal{B}$ intersects the boundary divisors of $\overline{\mathcal{M}}_{g}$ transversely outside $B$.
Let $\overline{\bm{f}}\colon \overline{\mathcal{S}}\to \mathcal{B}$ denote the corresponding family to the morphism $\mathcal{B}\to \mathcal{M}_{g}^{\mathrm{lci+}}$.
Replacing $(p\in \mathcal{B})$ with its \'{e}tale neighborhood, we may assume that there exists a smooth morphism $h\colon (p\in \mathcal{B})\to (0\in T)$ to a smooth curve such that $h^{-1}(0)=B$.
Hence we obtain a splitting family 
$$
\overline{\mathcal{S}}\xrightarrow{\overline{\bm{f}}} (p\in \mathcal{B})\xrightarrow{h} (0\in T)
$$
 of $\overline{f}$.
By Lemma~\ref{twosplitfam}~(2), replacing $(0\in T)$ with a finite covering, we obtain a splitting family $\mathcal{S}\to (p\in \mathcal{B})\to (0\in T)$ of $f$ such that any fiber of $\bm{f}_{t} \colon \mathcal{S}_{t}\to \mathcal{B}_{t}$ is stable with at most one node for each non-zero $t\in T$ sufficiently near $0$, whence the claim holds.
\end{proof}

\begin{lemma} \label{splitstable}
Let $\mathcal{M}_{g}^{\mathrm{Gor}}$ be the moduli stack defined in Example~\ref{*ex}~(2).
Let $\Gamma$ be the union of codimension one components of $\mathcal{M}_{g}^{\mathrm{Gor}}\setminus \overline{\mathcal{M}}_{g}$.
Let $\mathcal{A}$ be the set of isomorphism classes of stable fiber germs with at most one node.
Then for any fiber germ $f^{-1}(p)$ in $\mathcal{A}_{\infty}$,
the relative canonical model $\overline{f}$ belongs to $\mathcal{M}_{g}^{\mathrm{Gor}}\setminus \Gamma$.
\end{lemma}

\begin{proof}
Let $f\colon S\to (p\in B)$ be an element of $\mathcal{A}_{m}$.
Then the relative canonical model $\overline{f}$ belongs to $\mathcal{M}_{g}^{\mathrm{Gor}}$.
By the induction on $m$, we may assume that the claim holds for any element of $\mathcal{A}_{m-1}$.
Assume contrary that the central fiber of $\overline{f}$ belongs to $\Gamma$.
Let 
$$
\mathcal{S}\xrightarrow{\bm{f}} (p\in \mathcal{B})\xrightarrow{h} (0\in T)
$$
 be a splitting family of $f$ so that any fiber of $\bm{f}_{t} \colon \mathcal{S}_{t}\to \mathcal{B}_{t}$ belongs to $\mathcal{A}_{m-1}$ for each non-zero $t\in T$ sufficiently near $0$.
By the inductive assumption, the relative canonical model $\overline{\bm{f}}$ of $\bm{f}$ defines the moduli map $\rho_{\overline{\bm{f}}}\colon \mathcal{B}\to \mathcal{M}_{g}^{\mathrm{Gor}}$ such that the inverse image of $\Gamma$ consists of only one point $p$, which is of codimension two in $\mathcal{B}$.
This is a contradiction since $\Gamma$ is of codimension one in $\mathcal{M}_{g}^{\mathrm{Gor}}$.
\end{proof} 

Since smooth multiple fibers are atomic,
it is expected that smooth multiple curves of minimal fiber type forms some irreducible components of $\Gamma$ in Lemma~\ref{splitstable}.
However, the codimension of this locus is not known.
Moreover, $\Gamma$ may have other components generically parametrizing curves which are not of canonical type.
So we make the following conjecture, which is a variant of the Morsification conjecture:

\begin{conjecture}[Moduli conjecture] \label{moduliconj}
There exists a moduli stack $\mathcal{M}^{\star}_{g}$ of curves such that the following holds:

\begin{itemize}
\item[$(\mathrm{i})$]
The moduli stack contains all Gorenstein curves of genus $g$ of canonical fiber type as geometric points.

\item[$(\mathrm{ii})$]
Any divisorial component of the complement of $\mathcal{M}_{g}$ generically parametrizes stable curves with one node or smooth multiple curves of minimal fiber type.
\end{itemize}

\end{conjecture}

\begin{remark} \label{relMorsMod}
(1) For $g=2$, the Moduli conjecture holds true for $\mathcal{M}^{\star}_{2}=\mathcal{M}_{2}^{\mathrm{Gor}}\setminus \Gamma$ as in Lemma~\ref{splitstable}.
Indeed, the Morsification conjecture is true for $g=2$ (\cite{Hor3}) and there exist no multiple fibers of genus $2$.

\smallskip

\noindent
(2) The moduli stacks $\mathcal{M}^{\mathrm{Gor}}_{g}$ and $\mathcal{M}^{\mathrm{lci}}_{g}$ satisfy the condition~(i), but it is not known whether it satisfies the condition (ii).

\smallskip

\noindent
(3) If a moduli stack $\mathcal{M}^{\star}_{g}$ satisfies the condition (ii) and the following condition:

\begin{itemize}
\item[(iii)]
The moduli stack is smooth over the base field $k$.
\end{itemize}
(e.g., $\mathcal{M}_{g}$, $\overline{\mathcal{M}}_{g}$, $\mathcal{M}^{\mathrm{lci+}}_{g}$), then the Morsification conjecture holds for any fiber germ $f^{-1}(p)$ whose relative canonical model belongs to $\mathcal{M}^{\star}_{g}$ by the same proof of Proposition~\ref{splitcri}.
In particular, if $\mathcal{M}^{\star}_{g}$ satisfies (i), (ii) and (iii), then the Morsification conjecture holds true.
But this is no longer expected.

\smallskip

\noindent
(4) If we assume the Morsification conjecture, then the condition (ii) is equivalent to the  following weaker condition (ii)$'$:

\begin{itemize}
\item[(ii)$'$]
Any divisorial component of the complement of $\mathcal{M}_{g}$ generically parametrizes Gorenstein curves of canonical fiber type.
\end{itemize}
Indeed, we consider any splitting family 
$$
\overline{\mathcal{S}}\xrightarrow{\overline{\bm{f}}} (p\in \mathcal{B}) \xrightarrow{h} (0\in T)
$$
 of a relatively canonical fiber germ $\overline{f}\colon \overline{S}\to (p\in B)$ whose central fiber corresponds to a general point of a boundary prime divisor $\Delta$.
Then the fibration $\overline{\bm{f}}_{t}\colon \overline{\mathcal{S}}_{t}\to \mathcal{B}_{t}$ has a fiber corresponding to a point of $\Delta$ since the inverse image of $\Delta$ by the moduli map of $\overline{\bm{f}}$ is a horizontal curve on $\mathcal{B}$ with respect to $h$.
Hence we conclude that the divisor $\Delta$ generically parametrizes stable curves with one node or smooth multiple curves of minimal fiber type by the assumption of the Morsification conjecture and Lemma~\ref{twosplitfam}.
\end{remark}

\section{Slope inequality of fibered surfaces}
\label{sec:Slope inequality of fibered surfaces}

In this section, we fix an open substack $\mathcal{M}^{*}_{g}\subseteq \mathcal{M}^{\mathrm{Gor}}_{g}$ containing $\overline{\mathcal{M}}_{g}$
over an algebraically closed field $k$. 

\subsection{Tautological divisors}
It is well known that the $\Q$-Picard group of $\overline{\mathcal{M}}_{g}$ is generated by the Hodge class $\lambda$ and the boundary divisors $\delta_{0}, \delta_{1},\ldots, \delta_{\llcorner g/2 \lrcorner}$ freely ($g\ge 3$) or with one relation $10\lambda=\delta_{0}+2\delta_{1}$ ($g=2$) (\cite{Mor}).
Let $\kappa$ denote the first Morita--Mumford class on $\overline{\mathcal{M}}_{g}$.
Then Noether's formula 
$$
12\lambda=\kappa+\delta, \quad \delta:=\sum_{i=0}^{\llcorner g/2 \lrcorner}\delta_{i}
$$
holds.
We can extend $\lambda$ and $\kappa$ as functorial divisors on $\mathcal{M}^{*}_{g}$ as follows:

\begin{definition}[$\lambda$ and $\kappa$ as functorial divisors]
(1) First we recall the definition of Hodge bundle $\lambda$.
Let $\pi\colon \mathcal{U}^{*}_{g}\to \mathcal{M}^{*}_{g}$ 
denote the universal family on $\mathcal{M}^{*}_{g}$. 
Let $\omega_{\pi}$ denote the relative dualizing line bundle with respect to $\pi$ on the lisse-\'{e}tale site of $\mathcal{U}^{*}_{g}$.
By Lemma~\ref{relativedualizing}~(1), the pushforward $\pi_{*}\omega_{\pi}$ is a locally free sheaf of rank $g$.
The Hodge bundle $\lambda$ is defined as the determinant line bundle of $\pi_{*}\omega_{\pi}$.

For $g\ge 3$, the stack $\mathcal{M}^{*}_{g}$ has an open dense subscheme such as the locus of smooth curves with trivial automorphism groups.
Thus, the line bundle $\lambda$ is trivial over some open subscheme, and hence we can take a Cartier divisor $\Lambda$ on $\mathcal{M}^{*}_{g}$ such that $\lambda\cong \O_{\overline{\mathcal{M}}_{g}}(\Lambda)$
by Lemma~\ref{pdivCar}.
For the $g=2$ case, $\lambda^{\otimes 2}$ is trivial over some open subscheme
since the stabilizer of the generic point is of order two.
Thus, we can also regard $\lambda$ as a $\Q$-Cartier divisor on $\mathcal{M}^{*}_{g}$.
In any case, we can regard $\lambda$ as a functorial divisor on $\mathcal{M}^{*}_{g}$
as in Example~\ref{functdivex}~(1).

\smallskip

\noindent
(2) We regard $\kappa$ as a functional divisor on $\mathcal{M}^{*}_{g}$ as follows:
When $g\ge 3$, by the same argument as in (1), we can take two Cartier divisors $E$ and $F$ on $\mathcal{U}^{*}_{g}$
such that 
$$
\O_{\mathcal{U}^{*}_{g}}(E)\cong \omega_{\pi} \cong \O_{\mathcal{U}^{*}_{g}}(F)
$$
and the intersection of the supports of $E$ and $F$ is not dominant to $\mathcal{M}^{*}_{g}$.
The first Morita--Mumford class $\kappa$ is nothing but the divisor $\pi_{*}(E\cdot F)$ up to linear equivalence.
From Example~\ref{functdivex}~(2), 
 $\pi_{*}(E\cdot F)$ can be extended to a functorial divisor on $\mathcal{M}^{*}_{g}$.
For the $g=2$ case, the above argument also works by replacing $E$ and $F$ with suitable $\Q$-Cartier divisors.
\end{definition}

\begin{remark}
From the Grothendieck--Riemann--Roch theorem, we can show that $\kappa=\lambda^{(2)}-\lambda$ at any family $\mathcal{C} \to B$ in $\mathcal{M}^{\mathrm{lci}}_{g}$,
where $\lambda^{(2)}$ is the generalized Hodge class defined by the determinant of $\pi_{*}\omega_{\pi}^{\otimes 2}$.
Thus we can regard $\kappa$ as a $\Q$-Cartier divisor on $\mathcal{M}^{\mathrm{lci}}_{g}$ by using $\lambda^{(2)}-\lambda$ instead of $\kappa$.
\end{remark}



\begin{definition}[Horikawa divisors] \label{efftautodef}
We assume that $g\ge 3$ for simplicity (if $g=2$, the following argument can also be considered using only $\lambda$ and $\delta_{0}$ instead of $\lambda, \delta_{0},\ldots, \delta_{\llcorner \frac{g}{2} \lrcorner}$).
Let $D$ be a non-zero effective $\Q$-divisor on $\overline{\mathcal{M}}_{g}$ that does not contain any boundary divisor $\delta_{i}$.
Then we have
$$
D\sim_{\Q} a\lambda-\sum_{i=0}^{\llcorner \frac{g}{2} \lrcorner}b_{i}\delta_{i}
$$
for some rational numbers $a, b_{0},\ldots, b_{\llcorner g/2 \lrcorner}$.
It is well known that all the $a$ and $b_i$'s are positive. 
The {\em slope} of $D$ is defined as
$$
s_{D}:=\frac{a}{\min_{i}\{b_{i}\}}.
$$
Then we have
$$
\kappa-(12-s_{D})\lambda \sim_{\Q} \frac{1}{b}\left(D+\sum_{i=0}^{\llcorner \frac{g}{2} \lrcorner}(b_i-b)\delta_{i} \right)
$$
on $\overline{\mathcal{M}}_{g}$, where $b:=\min_{i}\{b_i\}$.

We now regard $\kappa$ and $\lambda$ as functorial divisors on $\mathcal{M}^{*}_{g}$.
Then, there exist a rational function $f$ on $\mathcal{M}_{g}^{*}$, a rational number $q$
and a $\Q$-divisor $\delta''$ supported on $\mathcal{M}_{g}^{*}\setminus \overline{\mathcal{M}}_{g}$ such that 
\begin{equation} \label{Hordiveq}
\left( \kappa-(12-s_{D})\lambda+q \mathrm{div}_{\mathcal{M}^{*}_{g}}(f)\right)_{\mathcal{M}_{g}^{*}} = \frac{1}{b}\left(D+\sum_{i=0}^{\llcorner \frac{g}{2} \lrcorner}(b_i-b)\delta_{i} \right)+\delta'',
\end{equation}
where the left hand side is the Weil divisor associated to the functorial divisor $\kappa-(12-s_{D})\lambda+q \mathrm{div}_{\mathcal{M}^{*}_{g}}(f)$, and $D$ and $\delta_{i}$'s are considered as divisors on $\mathcal{M}^{*}_{g}$ by taking their closure.
The functorial divisor 
$$
H_{D}:=\kappa-(12-s_{D})\lambda+q \mathrm{div}_{\mathcal{M}^{*}_{g}}(f)
$$
 is independent of the choices of $f$, $q$, and functorial divisors representing $\kappa$ and $\lambda$ satisfying \eqref{Hordiveq}.
We call $H_{D}$ the {\em Horikawa divisor associated to $D$} on $\mathcal{M}_{g}^{*}$.
\end{definition}

\subsection{Local invariants}
In this subsection, we define several local invariants of relatively minimal fiber germs.

\begin{definition}[Local invariants] \label{locinvfunct}
Let $D$ be a functorial divisor on $\mathcal{M}_{g}^{*}$.
Let $f\colon S\to (p\in B)$ be a relatively minimal fiber germ of genus $g$
such that the relative canonical model $\overline{f}\colon \overline{S}\to (p\in B)$ belongs to $\mathcal{M}_{g}^{*}$.
Assume that $f$ is {\em $D$-general}, that is, general fibers do not belong to the support of $D$.
Then, the {\em local invariant of $f^{-1}(p)$ associated to $D$}, denoted by $D(f^{-1}(p))$, is defined as the order of $D_{B}$ at $p$, where $D_{B}$ is the trace of $D$ at the moduli map of $\overline{f}$.
\end{definition}

The local invariant associated to a functorial divisor preserves under splitting deformations:

\begin{lemma} \label{locinvsplit}
Let $D$ be a functorial divisor on $\mathcal{M}_{g}^{*}$.

\smallskip

\noindent
$(1)$ Let $f\colon S\to (p\in B)$ be a relatively minimal $D$-general fiber germ whose relative canonical model belongs to $\mathcal{M}_{g}^{*}$.
Let $\mathcal{S}\xrightarrow{\bm{f}} (p\in \mathcal{B})\xrightarrow{h} (0\in T)$ be a splitting family of $f^{-1}(p)$.
Then, 
$$
D(f^{-1}(p))=\sum_{q\in \mathcal{B}_{t}}D(\bm{f}_{t}^{-1}(q))
$$ 
holds for any non-zero $t\in T$ sufficiently near $0$.

\smallskip

\noindent
$(2)$ Let $\mathcal{A}$ be a set of isomorphism classes of relatively minimal fiber germs of genus $g$ including all smooth fiber germs of genus $g$.
Assume that $D(f^{-1}(p))\ge 0$ for any $D$-general fiber germ $f^{-1}(p)$ in $\mathcal{A}$ whose relative canonical model belongs to $\mathcal{M}_{g}^{*}$.
Then, $D(f^{-1}(p))\ge 0$ holds for any $D$-general fiber germ $f^{-1}(p)$ in $\mathcal{A}_{\infty}$ whose relative canonical model belongs to $\mathcal{M}_{g}^{*}$.
\end{lemma}

\begin{proof}
The claim (2) follows from (1) by the definition of $\mathcal{A}_{\infty}$ (Definition~\ref{splittable}).
To prove (1), let $\overline{\bm{f}}\colon \overline{\mathcal{S}}\to (p\in \mathcal{B})$ denote the relative canonical model of $\bm{f}$.
From Lemma~\ref{relcanocomp}, it is a family of Gorenstein curves of genus $g$ of canonical fiber type and $\overline{\bm{f}}_{t}\colon \overline{\mathcal{S}}_{t}\to (p\in \mathcal{B}_{t})$ is isomorphic to the relative canonical model of $\bm{f}_{t}$ for each $t\in T$.
Since the natural immersion $\iota_{t}\colon \mathcal{B}_{t}\hookrightarrow \mathcal{B}$ is regular, we have $D_{\mathcal{B}_{t}}=\iota_{t}^{!}D_{\mathcal{B}}$.
Thus, the degree of $D_{\mathcal{B}_{t}}$ is independent of $t\in T$ sufficiently near $0$.
Hence, $D(f^{-1}(p))=\sum_{q\in \mathcal{B}_{t}}D(\bm{f}_{t}^{-1}(q))$ holds for any non-zero $t\in T$ sufficiently near $0$.
\end{proof}

\begin{example} \label{localinvex}
The following are typical local invariants, where we consider $\mathcal{M}^{*}_{g}=\mathcal{M}^{\mathrm{Gor}}_{g}$ in (1), (2), (3) below.

\smallskip

\noindent
(1) (Horikawa index) 
Let $D$ be a non-zero effective $\Q$-divisor on $\overline{\mathcal{M}}_{g}$ containing no boundary divisors (or $D=\delta_{1}$ and $g=2$).
Then, we call the local invariant $H_{D}(f^{-1}(p))$ the {\em Horikawa index of $f^{-1}(p)$ associated to $D$}, where $H_{D}$ is the Horikawa divisor associated to $D$ as in Definition~\ref{efftautodef}.
Note that a generically smooth fibration $f$ is $D$-general if and only if it is $H_{D}$-general.
By the definition of $H_D$, for a relatively minimal $D$-general fibered surface $f\colon S\to B$ of genus $g$, we have
$$
K_{f}^{2}=(12-s_{D})\chi_{f}+\sum_{p\in B}H_{D}(f^{-1}(p)).
$$
This is called the {\em slope equality} if $H_{D}(f^{-1}(p))$ is non-negative for any $p\in B$, since this is more precise than the slope inequality $K_{f}^{2}\ge (12-s_{D})\chi_{f}$.
We later show in Theorem~\ref{slopeeq} that it is true when all fiber germ $f^{-1}(p)$ satisfies the Morsification conjecture or the Moduli conjecture (Conjecture~\ref{moduliconj}) holds true.

\smallskip

\noindent
(2) (Local signature) 
Similarly to the definition of the Horikawa divisor $H_{D}$ in Definition~\ref{efftautodef},
we can define the {\em signature divisor} $\sigma_{D}$ as the functorial divisor
satisfying that the associated Weil divisor is $\Q$-linearly equivalent to $\kappa-8\lambda$ and a $\Q$-linear combination of $D$ and boundary divisors.
We call the number $\sigma_{D}(f^{-1}(p))$ the {\em local signature of $f^{-1}(p)$ associated to $D$}.
For a relatively minimal $D$-general fibered surface $f\colon S\to B$ of genus $g$, we have
$$
\Sign(S)=K_{f}^{2}-8\chi_{f}=\sum_{p\in B}\sigma_{D}(f^{-1}(p)),
$$
where $\Sign(S)$ is the signature of the intersection form on $H^{2}(S, \Q)$.

Ashikaga and Yoshikawa \cite{AsYo} introduced the signature divisor on $\overline{\mathcal{M}}_{g}$ and defined the local signature by using semistable reduction.
This local signature is equal to that we defined above (Remark~\ref{locsigncompare}).
By adopting our definition, the additivity of local signatures under splitting deformations is guaranteed (Lemma~\ref{locinvsplit}~(1)).

\smallskip

\noindent
(3) (Topological Euler contribution) 
By Noether's formula $12\lambda=\kappa+\delta$, there exist a rational function $f$ on $\mathcal{M}_{g}^{\mathrm{Gor}}$, a rational number $q$ and a $\Q$-divisor $\delta'$ supported on $\mathcal{M}_{g}^{\mathrm{Gor}}\setminus \overline{\mathcal{M}}_{g}$ such that 
\begin{equation}\label{Eulereq}
\left( 12\lambda-\kappa+q\mathrm{div}_{\mathcal{M}_{g}^{\mathrm{Gor}}}(f) \right)_{\mathcal{M}_{g}^{\mathrm{Gor}}}=\sum_{i}\delta_{i}+\delta'.
\end{equation}
The functorial divisor $12\lambda-\kappa+q\mathrm{div}_{\mathcal{M}_{g}^{\mathrm{Gor}}}(f)$ is independent of the choices of $f$, $q$ and functorial divisors representing $\kappa$ and $\lambda$ satisfying \eqref{Eulereq}.
We denote this functorial divisor by $\delta$ unless confusion arises.
It is easily seen that 
$$
\delta(f^{-1}(p))=\chi_{\mathrm{top}}(f^{-1}(p))-2+2g
$$
holds for a relatively minimal fiber germ $f\colon S\to (p\in B)$ of genus $g$ (see also the proof of Lemma~\ref{discrep}), which is called the {\em topological Euler contribution for $f^{-1}(p)$}.
For a relatively minimal fibered surface $f\colon S\to B$ of genus $g$, we have
$$
e_{f}=12\chi_{f}-K_{f}^{2}=\sum_{p\in B}\delta(f^{-1}(p)).
$$

\smallskip

\noindent
(4) (Virtual number of nodes of type $i$) 
Let $\mathcal{M}_{g}^{*}$ be the smooth locus of $\mathcal{M}_{g}^{\mathrm{Gor}}$.
We note that this contains $\mathcal{M}_{g}^{\mathrm{lci+}}$.
In this case, each boundary divisor $\delta_{i}$ is Cartier on $\mathcal{M}^{*}_{g}$ and hence can be regarded as a functorial divisor on $\mathcal{M}^{*}_{g}$.
The number $\delta_{i}(f^{-1}(p))$ is called the {\em virtual number of nodes of type $i$}.
Indeed, for a relatively minimal fiber germ $f\colon S\to (p\in B)$ of genus $g$ whose relative canonical model has reduced fibers, the number $\delta_{i}(f^{-1}(p))$ coincides with the number of nodes of type $i$ on the stable fibers split from $f^{-1}(p)$ by Lemma~\ref{locinvsplit}~(1) and Proposition~\ref{splitcri}.
\end{example}

\subsection{Slope inequality of Moriwaki type}
Combining Theorem~\ref{codimbound} with Moriwaki's relative Bogomolov inequality \cite{Mor2}, the following precise slope inequality holds:

\begin{theorem} \label{Moriwakiineq}
Let $f\colon S\to B$ be a relatively minimal fibered surface of genus $g$ over an algebraically closed field of characteristic $0$.
Assume that all the fibers of the relative canonical model of $f$ are reduced.
Then, we have
$$
K_{f}^{2}\ge \frac{4(g-1)}{g}\chi_{f}+\sum_{i=1}^{\llcorner g/2 \lrcorner}\frac{4i(g-i)-g}{g}\sum_{p\in B}\delta_{i}(f^{-1}(p)).
$$
\end{theorem}

\begin{proof}
By {\cite[Theorem~B]{Mor2}}, the divisor 
$$
D:=(8g+4)\lambda-g\delta_{0}-\sum_{i=1}^{\llcorner g/2 \lrcorner}4i(g-i)\delta_{i}
$$ 
on $\overline{\mathcal{M}}_{g}$ is weakly positive on $\mathcal{M}_{g}$.
Namely, for any point $x\in M_{g}$, any ample divisor $A$ on $\overline{M}_{g}$ and any positive rational number $\varepsilon$, the $\Q$-divisor $D+\varepsilon A$ on $\overline{M}_{g}$ is semiample at $x$, where $M_{g}$ and $\overline{M}_{g}$ are the coarse moduli spaces of $\mathcal{M}_{g}$ and $\overline{\mathcal{M}}_{g}$, respectively.
Note that the points and the rational Picard group of $\overline{\mathcal{M}}_{g}$ can be identified with those of $\overline{M}_{g}$, respectively.
Let $f\colon S\to B$ be a fibered surface in the claim, and let $\rho\colon B\to \mathcal{M}^{\mathrm{lci+}}_{g}$ be the moduli map of the relative canonical model of $f$.
By Noether's formula $12\lambda=\kappa+\delta$, it suffices to show that $\deg D_{B}\ge 0$,
where we regard $D$ as a functorial divisor on $\mathcal{M}^{\mathrm{lci+}}_{g}$.
We take a point $x$ in the intersection of the image of $\rho$ and $\mathcal{M}_{g}$.
Then there exists an effective $\Q$-divisor $E_{\varepsilon}$ on $\overline{\mathcal{M}}_{g}$ which is $\Q$-linearly equivalent to $D+\varepsilon A$ such that $x\notin E_{\varepsilon}$ since $D+\varepsilon A$ is semiample at $x$.
By Theorem~\ref{codimbound} and $\mathcal{M}^{\mathrm{lci+}}_{g}$ is smooth, $E_{\varepsilon}$ can be regarded as an effective $\Q$-Cartier $\Q$-divisor and hence a functorial divisor on $\mathcal{M}^{\mathrm{lci+}}_{g}$.
From $x\notin E_{\varepsilon}$ and Theorem~\ref{effthm}, the trace $E_{\varepsilon, B}$ is effective.
Taking the degree, we have
$$
\deg D_{B}+\varepsilon \deg A_{B}=\deg E_{\varepsilon, B} \ge 0.
$$
Hence we obtain $\deg D_{B}\ge 0$ by taking the limit as $\varepsilon$ approaches $0$.
\end{proof}

\begin{question}
Does Theorem~\ref{Moriwakiineq} hold in positive characteristic?
Can Theorem~\ref{Moriwakiineq} be generalized without assuming the reducibility of fibers?
\end{question}
Note that Theorem~\ref{Moriwakiineq} holds in positive characteristic if $f$ is semistable (\cite{Yam}).

\begin{remark}
The smoothness of $S$ is not needed in Theorem~\ref{Moriwakiineq}.
Indeed, the same inequality holds for any generically smooth family $f\colon S\to B$ of reduced, local complete intersection canonically polarized curves of genus $g$ over a smooth projective curve $B$, which is a generalization of {\cite[Proposition~4.1]{AFS}}.
\end{remark}


\subsection{Chern invariants}
In the rest of this section, we assume that the base field $k$ is of characteristic $0$ or consider in the complex analytic setting unless otherwise stated.
First we recall the Chern invariants of relatively minimal fiber germs $f\colon S\to (p\in B)$ introduced in \cite{TanII}.

Let $f'\colon S'\to B'$ be a semistable reduction of $f^{-1}(p)$ with a finite morphism $\varphi\colon B'\to B$ of degree $N$ totally ramified at $p$,
and consider the following commutative diagram:

$$
\xymatrix{
S' \ar[rrrd]_-{f'}  & \widetilde{S}  \ar[l]_-{\tau}  \ar[r]^-{\pi} & \widehat{S}  \ar[r]^-{\nu} & S\times_{B}B' \ar[r]^-{\psi} \ar[d]^-{f_{B'}} & S \ar[d]^-{f} \\
 & & & B' \ar[r]_-{\varphi} & B,
}
$$
where $\nu$ is the normalization of $S\times_{B}B'$, $\pi$ is the minimal resolution of $\widehat{S}$, and $\tau$ is the relatively minimal model of $\widetilde{S}$ over $B'$.
Thus, we can write
$$
K_{\widehat{S}}=\nu^{*}K_{S\times_{B}B'}-\Delta, \quad 
K_{\widetilde{S}}=\pi^{*}K_{\widehat{S}}-\sum_{i}a_{i}E_{i}, \quad
K_{\widetilde{S}}=\tau^{*}K_{S'}+\sum_{j}b_{j}F_{j}
$$
for some non-negative rational numbers $a_i$ and $b_j$,
where $\Delta$ is the conductor of $\nu$, $E_{i}$'s are $\pi$-exceptional curves and $F_{j}$'s are $\tau$-exceptional curves.
Thus, the difference 
\begin{equation} \label{anticanocycle}
(\psi\circ \nu \circ \pi)^{*}K_{f}-\tau^{*}K_{f'}=\pi^{*}\Delta+\sum_{i}a_{i}E_{i}+\sum_{j}b_{j}F_{j}
\end{equation}
is effective and supported on the central fiber of $\widetilde{S}\to B'$.
Moreover, there is a natural injection:
\begin{align*}
f'_{*}\omega_{f'} &\cong f'_{*} \tau_{*} \tau^{*}\omega_{f'} \\
&\hookrightarrow  (f_{B'}\circ \nu \circ \pi)_{*} \left( (\nu \circ \pi)^{*} \psi^{*}\omega_{f}\otimes \O_{\widetilde{S}}(-\lceil \pi^{*}\Delta \rceil) \right) \\
&\cong f_{B'*} \psi^{*}\omega_{f}\otimes \nu_{*}\O_{\widehat{S}}(-\Delta) \\
&\hookrightarrow f_{B'*}\psi^{*}\omega_{f} \\
&\cong \varphi^{*}f_{*}\omega_{f},
\end{align*}
which we denote by $\iota\colon f'_{*}\omega_{f'}\hookrightarrow \varphi^{*}f_{*}\omega_{f}$,
 where the first and third isomorphisms, the second injection, and the last isomorphism are respectively due to the projection formula, \eqref{anticanocycle} and Lemma~\ref{relativedualizing}~(1).
 
\begin{definition}[Chern invariants] \label{Cherndef}
The {\em Chern invariants} $c_{1}^2(f^{-1}(p))$, $c_{2}(f^{-1}(p))$ and $\chi_{f^{-1}(p)}$ are defined as the rational numbers

\begin{align*}
c_{1}^2(f^{-1}(p))&:=\frac{1}{N}((\psi\circ \nu \circ \pi)^{*}K_{f}+\tau^{*}K_{f'})\cdot ((\psi\circ \nu \circ \pi)^{*}K_{f}-\tau^{*}K_{f'}), \\
c_{2}(f^{-1}(p))&:=\delta(f^{-1}(p))-\frac{1}{N}\delta(f'^{-1}(p')), \\
\chi_{f^{-1}(p)}&:=\frac{1}{N}\length(\mathrm{Coker}(\iota)), 
\end{align*}
where $\delta(f^{-1}(p))=\chi_{\mathrm{top}}(f^{-1}(p))-2+2g$ is the topological Euler contribution of $f^{-1}(p)$ (Example~\ref{localinvex}~(3)).
\end{definition}

These invariants are independent of the choice of a semistable reduction, and satisfy the following (see \cite{TanII}):

\begin{itemize}[itemsep=3pt, parsep=0pt]

\item
(Positivity) $c_{1}^2(f^{-1}(p))$, $c_{2}(f^{-1}(p))$ and $\chi_{f^{-1}(p)}$ are non-negative, and one of them is zero if and only if
$f^{-1}(p)$ is semistable.
In this case, all of them are zero.

\item
(Noether's formula)
$$
12\chi_{f^{-1}(p)}=c_{1}^2(f^{-1}(p))+c_{2}(f^{-1}(p)).
$$

\item
(Local-to-global formula)
Let $f\colon S\to B$ be a relatively minimal fibered surface of genus $g$.
Let $f'\colon S'\to B'$ be a semistable reduction of $f$ with a finite morphism $B'\to B$ of degree $N$.
Then,
\begin{align*}
K_f^2-\frac{1}{N}K_{f'}^2&=\sum_{p\in B}c_1^2(f^{-1}(p)), \\
e_f-\frac{1}{N}e_{f'}&=\sum_{p\in B}c_2(f^{-1}(p)), \\
\chi_f-\frac{1}{N}\chi_{f'}&=\sum_{p\in B}\chi_{f^{-1}(p)}.
\end{align*}

\item
(Miyaoka--Yau type inequality {\cite[Theorem~3.4, Proposition~3.5]{TanII}})
$$
c_{1}^2(f^{-1}(p))\le 8\chi_{f^{-1}(p)}
$$ 
with equality holding if and only if $f^{-1}(p)$ is semistable or a multiple of a curve with at worst ordinary nodes.
Moreover, if the canonical model $\overline{f}^{-1}(p)$ is reduced, then
$$
c_{1}^2(f^{-1}(p))\le 6\chi_{f^{-1}(p)}.
$$ 
\end{itemize}

\begin{remark} \label{Chernrem}
(1) Our definition of Chern invariants in Definition~\ref{Cherndef} is slightly different from Tan's original definiton.
The coincidence can be checked directly or using similar arguments of the proof of  Lemma~\ref{discrep} below.



\smallskip

\noindent
(2) In the complex analytic setting, Ashikaga \cite{Ash} introduced the {\em local signature defect} of a fiber germ $f^{-1}(p)$, which is a $\Q$-valued invariant and denoted by $\mathrm{Lsd}(f^{-1}(p))$.
This invariant satisfies the local-to-global formula 
$$
\Sign(S)-\frac{1}{N}\Sign(S')=\sum_{p\in B}\mathrm{Lsd}(f^{-1}(p))
$$
for $f$ and $f'$ above, where $\Sign(S)$ is the signature of the intersection form on $H^{2}(S, \Q)$.
By using a partial semistable reduction and the local-to-global formula,
we can show that 
\begin{equation} \label{Lsd}
\mathrm{Lsd}(f^{-1}(p))=\frac{1}{3}(c_{1}^2(f^{-1}(p))-2c_2(f^{-1}(p)))=c_{1}^2(f^{-1}(p))-8\chi_{f^{-1}(p)}
\end{equation}
holds (see also the proof of Lemma~\ref{discrep}). 
In particular, $\mathrm{Lsd}(f^{-1}(p))$ is non-positive by Miyaoka--Yau type inequality.

\smallskip

\noindent
(3) All Chern invariants and the local signature defect are topological invariants of fiber germs in the complex analytic setting:
In fact, $\mathrm{Lsd}(f^{-1}(p))$ and $c_{2}(f^{-1}(p))$ can be written by the datum of the topological monodromy of $f^{-1}(p)$ by definition,
and $c_{1}^{2}(f^{-1}(p))$ and $\chi_{f^{-1}(p)}$ are determined by these invariants by Noether's formula and \eqref{Lsd}.
\end{remark}

Let $f\colon S\to B$ be a relatively minimal fibration of genus $g$ over a smooth curve $B$.
Let $f'\colon S'\to B'$ be a semistable reduction of $f$ defined by a finite morphism $\varphi\colon B'\to B$ of degree $N$.
The relative canonical models of $f$ and $f'$ correspond to the moduli maps $\rho\colon B\to \mathcal{M}^{\mathrm{Gor}}_{g}$ and $\rho'\colon B'\to \overline{\mathcal{M}}_{g}$, respectively.
The diagram
$$
\xymatrix{
B' \ar[r]^-{\varphi} \ar[d]_-{\rho'}  & B \ar[d]^-{\rho}  \\
\overline{\mathcal{M}}_{g} \ar@{}[r]|*{\subseteq} & \mathcal{M}^{\mathrm{Gor}}_{g}.
}
$$
is not commutative over the points $p\in B$ where $f^{-1}(p)$ is not semistable.
From the point of view of the functorial divisors $\lambda$, $\delta$ and $\kappa$ on $\mathcal{M}^{\mathrm{Gor}}_{g}$, the local-to-global formula can be interpreted as follows:

\begin{lemma} \label{discrep}
In the above situations, we have
\begin{align*}
\kappa_{B}-\frac{1}{N}\varphi_{*}\kappa_{B'}&=\sum_{p\in B}c_1^2(f^{-1}(p))p, \\
\delta_{B}-\frac{1}{N}\varphi_{*}\delta_{B'}&=\sum_{p\in B}c_2(f^{-1}(p))p, \\
\lambda_{B}-\frac{1}{N}\varphi_{*}\lambda_{B'}&=\sum_{p\in B}\chi_{f^{-1}(p)}p, 
\end{align*}
where $D_{B}$ and $D_{B'}$ are respectively the traces at $\rho$ and $\rho'$ for a functorial divisor $D$.
\end{lemma}

\begin{proof}
Note that the left hand sides in the claim do not depend on the choices of functorial divisors in the linear equivalence classes and on the choices of semistable reductions.
Restricting to a fiber germ at a closed point $p\in B$, we may assume that any fiber of $f$ other than $f^{-1}(p)$ is smooth.
If $f^{-1}(p)$ is semistable, then both the left and right hand sides are zero.
Thus the claim holds when all fibers of $f$ are semistable.
We assume that $f^{-1}(p)$ is not semistable.
We take an embedding of $f$ to a fibered surface $f_{1}\colon S_{1}\to B_{1}$ with $B_{1}$ complete arbitrarily, and take a partial semistable reduction $f'_{1}\colon S'_{1}\to B'_{1}$ at the fibers other than $f^{-1}(p)$ with a finite morphism $\varphi_{1}\colon B'_{1}\to B_{1}$.
Since $\varphi_{1}$ is \'{e}tale at $p$, the pullback of the left hand side and the right hand side in the claim over $B$ coincides with that over $B'_{1}$, respectively.
By applying the local-to-global formula to $f'_{1}$, the claim for $f'_{1}$ holds.
Hence the claim for $f$ holds by \'{e}tale descent.
\end{proof}

By Lemma~\ref{discrep}, the Horikawa index of $f^{-1}(p)$ can be described by using the Horikawa index of the semistable reduction and Chern invariants:

\begin{proposition} \label{Horindcompare}
Let $D$ be as in Example~\ref{localinvex}~(1), and let $f\colon S\to B$ be a relatively minimal $D$-general fibration over a smooth curve $B$.
Let $f'\colon S'\to B'$ be a semistable reduction of $f$ defined by a finite morphism $\varphi\colon B'\to B$ of degree $N$.
Then for any closed point $p\in B$, we have
$$
H_{D}(f^{-1}(p))=\frac{1}{N}\sum_{q\in \varphi^{-1}(p)}H_{D}(f'^{-1}(q))+c_1^{2}(f^{-1}(p))-(12-s_{D})\chi_{f^{-1}(p)}.
$$
In particular, if $H_{D}(f^{-1}(p))\ge 0$ and all fibers of the relative canonical model of $f'$ over $\varphi^{-1}(p)$ are not contained in $D$ and $\delta_{i}$ with $b_{i}>b$ in \eqref{Hordiveq}, we have
$$
c_1^{2}(f^{-1}(p))\ge (12-s_{D})\chi_{f^{-1}(p)}.
$$
\end{proposition}

\begin{proof}
The first assertion is due to Lemma~\ref{discrep} and the definition of the Horikawa index.
The last assertion follows from the first assertion, $H_{D}(f^{-1}(p))\ge 0$ and $H_{D}(f'^{-1}(q))=0$ for
 each $q\in \varphi^{-1}(p)$.
\end{proof}

\begin{remark} \label{locsigncompare}
Similarly to Proposition~\ref{Horindcompare}, we can also show 
$$
\sigma_{D}(f^{-1}(p))=\frac{1}{N}\sum_{q\in \varphi^{-1}(p)}\sigma_{D}(f'^{-1}(q))+\mathrm{Lsd}(f^{-1}(p)).
$$
The right hand side is nothing but the local signature introduced in \cite{AsYo}.
\end{remark}

\subsection{Slope equality of fibered surfaces}
In this subsection, we establish the slope equalities for general fibered surfaces, the main theorem in this paper:

\begin{theorem}[Slope equality] \label{slopeeq}
Let $D$ be a non-zero effective $\Q$-divisor on $\overline{\mathcal{M}}_{g}$ which contains no boundary divisors $($or $D=\delta_{1}$ and $g=2$$)$ with $s_{D}\ge 4$.
Let $f\colon S\to (p\in B)$ be a relatively minimal $D$-general fiber germ of genus $g$.
Assume that $f^{-1}(p)$ satisfies the Morsification conjecture or the Moduli conjecture (Conjecture~\ref{moduliconj}) holds true.
Then, the Horikawa index $H_{D}(f^{-1}(p))$ is non-negative.
In particular, for any relatively minimal $D$-general fibered surface $f\colon S\to B$ of genus $g$ assuming that all fibers of $f$ satisfy the Morsification conjecture or the Moduli conjecture holds true, we have 
$$
K_{f}^{2}=(12-s_{D})\chi_{f}+\sum_{p\in B}H_{D}(f^{-1}(p))\ge (12-s_{D})\chi_{f}.
$$
\end{theorem}

\begin{proof}
First we prove the claim when $f^{-1}(p)$ satisfies the Morsification conjecture.
Applying Lemma~\ref{locinvsplit} to the Horikawa divisor $H_{D}$, it is enough to show that the Horikawa index $H_{D}(f^{-1}(p))$ is non-negative for any $D$-general fiber germ $f\colon S\to (p\in B)$ which is stable with at most one node or smooth multiple.
It follows that $H_{D}(f^{-1}(p))\ge 0$ for any semistable $D$-general fiber germ $f^{-1}(p)$ since $H_{D}$ is effective on $\overline{\mathcal{M}}_{g}$.
Hence, it suffices to consider only smooth multiple fiber germs.
Let $f\colon S\to (p\in B)$ be a smooth multiple $D$-general fiber germ of multiplicity $m$.
We take a semistable reduction $f'\colon S'\to (p'\in B')$ of $f$ defined by a finite morphism $\varphi\colon (p'\in B')\to (p\in B)$.
Note that $\varphi$ can be taken as a cyclic covering of degree $m$ totally branched at $p$ and $S'$ is obtained by the normalization of $S\times_{B}B'$.
Since $f'^{-1}(p')$ is smooth, we have $H_{D}(f'^{-1}(p'))\ge 0$.
On the other hand, it follows from Proposition~\ref{Horindcompare} 
that 
$$
H_{D}(f^{-1}(p))=\frac{1}{m} H_{D}(f'^{-1}(p'))+c_1^{2}(f^{-1}(p))-(12-s_{D})\chi_{f^{-1}(p)}.
$$
Since $f^{-1}(p)$ is smooth multiple, the equality $c_1^{2}(f^{-1}(p))=8\chi_{f^{-1}(p)}$ holds  in the Miyaoka--Yau type inequality.
Hence 
$$
H_{D}(f^{-1}(p))=\frac{1}{m} H_{D}(f'^{-1}(p'))+(s_{D}-4)\chi_{f^{-1}(p)}\ge 0
$$
from the positivity of $\chi_{f^{-1}(p)}$ and the assumption $s_{D}\ge 4$.

To prove the claim under the assumption of the Moduli conjecture (Conjecture~\ref{moduliconj}),
it suffices to show that $H_{D}$ is effective as a functorial divisor on the moduli stack $\mathcal{M}^{\star}_{g}$ as in Conjecture~\ref{moduliconj}.
Assume contrary that $H_{D}$ is not effective.
Then, Theorem~\ref{effthm} implies that the associated Weil divisor $(H_{D})_{\mathcal{M}^{\star}_{g}}$ is not effective.
Namely, there exists an irreducible component $\Delta$ of the complement $\mathcal{M}^{\star}_{g}\setminus \overline{\mathcal{M}}_{g}$ such that the coefficient of $\Delta$ in $(H_{D})_{\mathcal{M}^{\star}_{g}}$ is negative.
The condition (ii) in Conjecture~\ref{moduliconj} implies that $\Delta$ generically parametrizes smooth multiple curves of minimal fiber type.
We take a smooth multiple curve $C$ corresponding to a general point of $\Delta$ and a very good smoothing family $f\colon S\to (p\in B)$ of $C$.
Since the closure of $D$ does not contain $\Delta$, we may assume that $f$ is $D$-general.
Then, we have
$$
H_{D}(f^{-1}(p))=\deg (H_{D})_{B}< 0
$$ 
since the coefficient of $\Delta$ in $(H_{D})_{\mathcal{M}^{\star}_{g}}$ is negative and the image of $\rho_{f}$ intersects the support of $H_{D}$ at $C\in \Delta$.
On the other hand, since $f^{-1}(p)$ is smooth multiple, $H_{D}(f^{-1}(p))$ is non-negative as shown above, which is a contradiction.
\end{proof}

\begin{remark}
Theorem~\ref{slopeeq} in positive characteristic also holds if we further assume in the statement of the Morsification conjecture and the Moduli conjecture that 
the multiplicities of smooth multiple curves are not divisible by the characteristic of $k$.
Indeed, for a smooth multiple fiber germ $f^{-1}(p)$ of multiplicity $m$ not divisible by the characteristic of $k$, we can take a semistable reduction of $f^{-1}(p)$ by a cyclic covering of degree $m$ over $B$, and compute directly 
$$
c_{1}^{2}(f^{-1}(p))=8\chi_{f^{-1}(p)}=4\left(1-\frac{1}{m}\right)(g-1),
$$
where the Chern invariants of $f^{-1}(p)$ are similarly defined as in characteristic $0$ if we can take a semistable reduction of $f^{-1}(p)$ by a separable finite base change.
Then, the proof of Theorem~\ref{slopeeq} also works in this setting.
In particular, the slope equality for $D$-general fibered surfaces whose canonical model has reduced fibers holds in any characteristic.
Note that the claim for the reduced fiber case also follows from the effectiveness of the Horikawa divisor $H_{D}$ on the moduli stack $\mathcal{M}^{\mathrm{lci+}}_{g}$ due to Theorem~\ref{codimbound}.
\end{remark}

\begin{remark}
In the complex analytic setting, the non-negativity of the Horikawa index also holds if $f$ satisfies the Morsification conjecture in the analytic setting or the Moduli conjecture holds true.
Indeed, the Horikawa index can be defined in the same way as in the algebraic setting by replacing $\mathcal{M}^{\mathrm{Gor}}_{g}$ with its analytification.
Thus the same proof of Theorem~\ref{slopeeq} works in the complex analytic setting.
\end{remark}

\begin{remark}
Historically, the slope equality was first established by Horikawa \cite{Hor2} for genus $2$ fibrations ($g=2$ and $D=\delta_{1}$ in Theorem~\ref{slopeeq}) and the term ``Horikawa index'' originates from this result.
Later, Reid \cite{Rei} established the slope equality for non-hyperelliptic genus $3$ fibrations ($g=3$ and $D$ is the hyperelliptic locus in Theorem~\ref{slopeeq}).
The slope equality in Theorem~\ref{slopeeq} was known for the cases where $g\le 5$ and $D$ achieves the bound $s_{g}$ and where $g$ is odd and $D$ is the Brill--Noether divisor $BN_{g,d}^{1}$ (\cite{Kon}, {\cite[Appendix]{AsYo}}).
For more details of slope equalities, see the survey \cite{AsKo2}.
Note that our Horikawa index $H_{D}(f^{-1}(p))$ preserves splitting deformations by Lemma~\ref{locinvsplit}, and hence the question in {\cite[p. 32]{AsKo2}} automatically holds.
\end{remark}

\subsection{Decomposition of Horikawa index}

Let $D$ be as in Example~\ref{localinvex}~(1).
According to Definition~\ref{efftautodef}, The Horikawa divisor $H_{D}$ is of the form
$$
H_{D}=\frac{1}{b}\left(D+\sum_{i=0}^{\llcorner \frac{g}{2} \lrcorner}(b_i-b)\delta_{i} \right)+\delta''
$$
as a Weil divisor.
Thus we have
\begin{equation} \label{Hortypei}
H_{D}(f^{-1}(p))=\frac{1}{b}D(f^{-1}(p))+\frac{b_i-b}{b}
\end{equation}
 for any stable $D$-general fiber germ $f^{-1}(p)$ with one node of type $i$,
 where $D$ is assumed to be a functorial divisor on $\overline{\mathcal{M}}_{g}$.
For a smooth multiple $D$-general fiber germ $f^{-1}(p)=mC$ of multiplicity $m$, 
we can compute the Horikawa index $H_{D}(f^{-1}(p))$ by using Proposition~\ref{Horindcompare}.
Indeed, by taking the semistable reduction $f'\colon S'\to (p'\in B')$ defined by a cyclic covering $(p'\in B')\to (p\in B)$ of degree $m$ totally branched at $p$, we obtain
$$
H_{D}(f^{-1}(p))=\frac{1}{m} H_{D}(f'^{-1}(p'))+(s_{D}-4)\chi_{f^{-1}(p)}
$$
as in the proof of Theorem~\ref{slopeeq}.
Let $g(C)$ denote the genus of $C$,
which satisfies $g-1=m(g(C)-1)$.
By the definition of $c_{2}(f^{-1}(p))$ and the equality of the Miyaoka--Yau type inequality,
we have 
$$
c_{2}(f^{-1}(p))=2-2g(C)-(2-2g)=\left(1-\frac{1}{m} \right)(2g-2),
$$
$$
\chi_{f^{-1}(p)}=\frac{1}{4}c_{2}(f^{-1}(p))=\frac{1}{2}\left(1-\frac{1}{m} \right)(g-1).
$$
Since $f'^{-1}(p')$ is smooth,
 $H_{D}(f'^{-1}(p'))$ can be identified with the intersection number of $D$ and the image of the moduli map of $f'$ multiplied by $b^{-1}$.
Hence we obtain
\begin{equation} \label{Horsmmult}
H_{D}(f^{-1}(p))=\frac{1}{mb} D(f'^{-1}(p'))+\frac{s_{D}-4}{2}\left(1-\frac{1}{m} \right)(g-1).
\end{equation}
Therefore, if a $D$-general fiber germ $f^{-1}(p)$ satisfies the Morsification conjecture, 
then the Horikawa index $H_{D}(f^{-1}(p))$ can be described by a non-negative linear combination of \eqref{Hortypei} and \eqref{Horsmmult} by Lemma~\ref{locinvsplit}~(1).

\begin{definition} \label{smmulttype}
Let $f\colon S\to (p\in B)$ be a $D$-general smooth multiple fiber germ of genus $g$.
Then $f$ is {\em $D$-indecomposable} if for any splitting family $\mathcal{S}\to (p\in \mathcal{B})\to (0\in T)$ of $f$, the fibration $\bm{f}_{t}\colon \mathcal{S}_{t}\to \mathcal{B}_{t}$ has no smooth fibers belonging to $D$ for non-zero $t\in T$ sufficiently near $0$.
In the view of the Morsification conjecture and the Moduli conjecture, it is expected to be equivalent to that the moduli point $f^{-1}(p)$ is not contained in the closure of $D$ on a suitable moduli stack $\mathcal{M}^{\star}_{g}$.

For $n\ge 0$ and $m\ge 2$, the fiber germ $f^{-1}(p)$ is said to be {\em of type $(m,n)$} if the multiplicity of $f^{-1}(p)$ equals to $m$ and $D(f'^{-1}(p'))=n$, where $f'^{-1}(p')$ is the semistable reduction of $f$ defined by the cyclic covering $(p'\in B')\to(p\in  B)$ of degree $m$ totally branched at $p$.
Let $\mathcal{A}^{D}_{(m,n)}$ denote the set of isomorphism classes of $D$-indecomposable smooth multiple fiber germs of type $(m, n)$ and smooth fiber germs of genus $g$.
\end{definition}

\begin{proposition} \label{smmultsplit}
Let $D$ be a non-zero effective $\Q$-divisor on $\overline{\mathcal{M}}_{g}$ which contains no boundary divisors, and let $m\ge 2$ be an integer such that $g-1\in m\Z$.
Let $Z$ denote the closed substack of $\mathcal{M}_{g}$ parametrizing curves having automorphisms of order $m$ with no fixed points.
Then the following holds:

\smallskip

\noindent
$(1)$ If $Z$ is not contained in $D$, then any $D$-general smooth multiple fiber germ of multiplicity $m$ is splittable into fibers in $\mathcal{A}^{D}_{(m,0)}$.

\smallskip

\noindent
$(2)$ If $Z$ is contained in $D$, then there exist finite positive numbers $n_1,\ldots, n_{l}$ such that any $D$-general smooth multiple fiber germ of multiplicity $m$ belongs to $(\cup_{i=1}^{k}\mathcal{A}^{D}_{(m,n_{i})})_{\infty}$.

\smallskip

\noindent
$(3)$ If $m=2$, $Z$ is contained in $D$ and $D$ is non-singular at the generic point of $Z$, then any $D$-general smooth multiple fiber germ of multiplicity $m$ is splittable into fibers in $\mathcal{A}^{D}_{(m,1)}$.
\end{proposition}

\begin{proof}
Note that $Z$ is irreducible, smooth and has codimension at least two (cf.\ {\cite[Theorem~(5.15)]{DeMu}}).
Let $f\colon S\to (p\in B)$ be a $D$-general smooth multiple fiber germ of multiplicity $m$.
Let $f'\colon S'\to B'$ be the semistable reduction defined by the cyclic covering $(p'\in B')\to (p\in B)$ of degree $m$ totally branched at $p$, and let $G\cong \Z/m\Z$ be its Galois group.
Note that $S'$ is the normalization of the fiber product $S\times_{B}B'$, the action of $G$ on $S'$ has no fixed points and $f'$ is $G$-equivariant and smooth.
Let $\rho_{f'}\colon B'\to \mathcal{M}_{g}$ denote the moduli map of $f'$.

\noindent
{\bf Step~1.}
Let $\pi\colon \mathcal{C}\to (x\in X)$ be a standard Kuranishi family for $f'^{-1}(p')$ ({\cite[Chapter~XI, \S6, Definition~(6.7)]{ACG}}).
Then $G$ acts on $\pi\colon \mathcal{C}\to X$ extending the action on $\pi^{-1}(x)\cong f'^{-1}(p')$.
Note that the moduli map $X\to \mathcal{M}_{g}$ is \'{e}tale, and the pullback $Z_{X}$ of $Z$ equals the locus of $G$-fixed points of $X$.
Shrinking $B$ to a neighbourhood of $p$, there exists a $G$-equivariant morphism $B'\to X$ by the universal property of the Kuranishi family.
Taking the quotient by $G$, we obtain a morphism $B\to X/G$.
Now we take a resolution $\alpha\colon Y\to X/G$ of the cyclic quotient singularities.
Note that $X/G$ is equisingular along the image of $Z_{X}$ and hence we can take $\alpha$ as the successive weighted blowing-ups along the centers dominating the image of $Z_{X}$ (cf.\ \cite{Fuj}).
Since the image of $B\to X/G$ is not contained in the image of $Z_X$, the morphism $B\to X/G$ uniquely lifts to $B\to Y$.
Let $Y'$ denote the normalization of the fiber product $Y\times_{X/G}X$.
Then we have a $G$-equivariant morphism $B'\to Y'$ which is a lift of $B'\to X$ with the quotient $B\to Y$.
The effective divisor $D_{Y'}$ can be decomposed into 
$$
D_{Y'}=\widehat{D}_{X}+\sum_{i}n_{i}E_{i},
$$
where $\widehat{D}_{X}$ is the proper transform of $D_{X}$ and $E_{i}$'s are prime divisors which are exceptional over $X$.
Note that the coefficients $n_{i}$ depend only on $D$ and $Z$, and are independent of the choice of a fiber germ $f^{-1}(p)$.
Moreover, if $Z$ is not contained in $D$, then we have $D_{Y'}=\widehat{D}_{X}$ by the assumption of the resolution $\alpha$.
We embed $B$ as the graph $B\subseteq B\times Y$ of $B\to Y$ and take a sufficiently small open neighbourhood $U$ of $p$ in $B\times Y$.
Replacing $B$ with $B\cap U$, we obtain a closed immersion $B\subseteq U$.
The fiber product $U':=U\times_{Y}Y'$ is a $G$-invariant open subvariety of $B\times Y'$ and we have a $G$-equivariant closed immersion $B'\subseteq U'$ with the quotient $B\subseteq U$.

\noindent
{\bf Step~2.}
Next we take $3g-4$ general hypersurfaces $H_{1},\ldots, H_{3g-4}$ of $U$ containing $B$ and put 
$$
\mathcal{B}:=H_{1}\cap\cdots \cap H_{3g-4}
$$
and $\mathcal{B}':=\mathcal{B}\times_{U} U'$.
Replacing $U$ with a neighbourhood of $p$, we may assume that there is a smooth morphism $h\colon (p\in \mathcal{B})\to (0\in T)$ to a germ of a smooth curve such that $h^{-1}(0)=B$.
Let $\bm{f}'\colon \mathcal{S}'\to \mathcal{B}'$ denote the family of smooth curves which is the pullback of the Kuranishi family $\pi\colon \mathcal{C}\to X$.
By construction, $G$ acts on $\mathcal{S}'$ and $\mathcal{B}'$ equivariantly, and there are no $G$-fixed points on $\mathcal{S}'$.
Let $\bm{f}\colon \mathcal{S}\to \mathcal{B}$ be the quotient of $\bm{f}'$ by $G$.
Then we obtain the splitting family of $f$
$$
\mathcal{S}\xrightarrow{\bm{f}} (p\in \mathcal{B}) \xrightarrow{h} (0\in T).
$$
Since the smooth multiple fiber germ $f$ is atomic, 
each family $\bm{f}_{t}\colon \mathcal{S}_{t}\to \mathcal{B}_{t}$ has exactly one singular fiber $\bm{f}_{t}^{-1}(p_{t})$ which is smooth multiple of degree $m$, and hence the branch locus of $\mathcal{B}'\to \mathcal{B}$ is a section of $h$.
The families $\bm{f}'_{t}\colon \mathcal{S}'_{t}\to \mathcal{B}'_{t}$ are semistable reductions of $\bm{f}_{t}$ by the cyclic covering $(p'_{t}\in \mathcal{B}'_{t})\to (p_{t}\in \mathcal{B}_{t})$ of degree $m$ totally branched at $p_{t}$ and $\bm{f}'_{0}\colon \mathcal{S}'_{0}\to (p'_{0}\in \mathcal{B}'_{0})$ is isomorphic to $f'\colon S'\to (p'\in B')$.

\noindent
{\bf Step~3.}
Assume that $Z$ is not contained in $D$.
Then the divisor $D_{\mathcal{B}'}$ does not contain the locus of $G$-fixed points on $\mathcal{B}'$ since $H_{i}$'s are general.
Then the fiber $\bm{f}_{t}^{'-1}(p'_{t})$ does not belong to $D$ and hence the claim (1) holds.

Assume that $Z$ is contained in $D$.
Then the divisor $D_{\mathcal{B}'}$ may contain the locus $Z_{\mathcal{B}'}$ of $G$-fixed points on $\mathcal{B}'$.
In this case, the coefficient of $Z_{\mathcal{B}'}$ in $D_{\mathcal{B}'}$ equals $n_{i}$ for some $i$ and we can write 
$$
D_{\mathcal{B}'}=D'+n_{i}Z_{\mathcal{B}'},
$$
where $D'$ is an effective divisor on $\mathcal{B}'$ which does not contain $Z_{\mathcal{B}'}$.
To prove the claim (2), we may assume that $f$ is $D$-indecomposable by performing finitely many splitting deformations.
Then it follows that $D'=0$ and $D(\bm{f}_{t}^{'-1}(p'_{t}))=n_{i}$.
Since the numbers $n_{i}$ are determined only on $Z$ and $D$, the claim (2) holds.
If moreover $m=2$ and $D$ is non-singular at the generic point of $Z$, then the resolution $\alpha \colon Y\to X/G$ is obtained by one blowing-up along the image of $Z_{X}$ and $D_{Y'}=\widehat{D}_{X}+E$ holds, where $E$ is an exceptional prime divisor.
Then the claim (3) holds by the above argument.
\end{proof}

The following corollary says that the Morsification conjecture implies ``algebraic Morsification'':

\begin{corollary} \label{Hordecomp}
Let $D$ be as in Proposition~\ref{smmultsplit} with the irreducible decomposition $D=\sum_{i}m_{i}D_{i}$.
Let $f\colon S\to (p\in B)$ be a relatively minimal $D$-general fiber germ of genus $g$ which satisfies the Morsification conjecture.
Then, there exist finite non-negative numbers $n_1,\ldots, n_{l}$ depending only on $D$  such that $f^{-1}(p)$ is splittable by finitely many splitting deformations into fiber germs $\bm{f}_{t}^{-1}(p_{t})$ of the following types:

\smallskip

\noindent 
$(\mathrm{i})$
Smooth fiber germs which do not belong to $D$ as moduli points.
In this case, 
$$
H_{D}(\bm{f}_{t}^{-1}(p_{t}))=0.
$$

\smallskip

\noindent 
$(\mathrm{ii})$
Smooth fiber germs whose moduli maps meet the support of $D$ transversely at a smooth point of $D_i$.
In this case, 
$$
H_{D}(\bm{f}_{t}^{-1}(p_{t}))=\frac{m_i}{b}.
$$

\smallskip

\noindent 
$(\mathrm{iii})$
Stable fiber germs with one node of type $i$ which do not belong to $D$ as moduli points.
In this case, 
$$
H_{D}(\bm{f}_{t}^{-1}(p_{t}))=\frac{b_{i}-b}{b}.
$$

\smallskip

\noindent 
$(\mathrm{iv})$
$D$-indecomposable smooth multiple fiber germs of type $(m, n_{i})$.
In this case, 
$$
H_{D}(\bm{f}_{t}^{-1}(p_{t}))=\frac{n_{i}}{mb}+\frac{s_{D}-4}{2}\left(1-\frac{1}{m} \right)(g-1).
$$
\end{corollary}

\begin{proof}
Let $f^{-1}(p)$ be a $D$-general relatively minimal fiber germ of genus $g$.
By the assumption of the Morsification conjecture, we may assume that it is a stable fiber germ with at most one node or a smooth multiple fiber germ.
If it is stable, then it is splittable into fibers in (i), (ii) and (iii) by the same proof of Proposition~\ref{splitcri}.
If it is smooth multiple, then Proposition~\ref{smmultsplit} implies that it is splittable into fibers in (i), (ii) and (iv).
\end{proof}

According to Corollary~\ref{Hordecomp} and Lemma~\ref{locinvsplit}~(1), we can write by abuse of notation that
\begin{equation}\label{Hordivdecomp}
H_{D}=\sum_{i}\frac{m_i}{b}D_i+\sum_{i=0}^{\llcorner \frac{g}{2} \lrcorner}\frac{b_i-b}{b}\delta_{i} +\sum_{i=1}^{k} \left(\frac{n_{i}}{mb}+\frac{s_{D}-4}{2}\left(1-\frac{1}{m} \right)(g-1)\right)\delta_{(m, n_{i})}.
\end{equation}
Here, the meaning of the equation is as follows: 
If the fiber germ $f^{-1}(p)$ splits into fibers listed in Corollary~\ref{Hordecomp} by finite splitting deformations $\xi$, then 
we let $D_{i}(f^{-1}(p), \xi)$, $\delta_{i}(f^{-1}(p), \xi)$ and $\delta_{(m, n_{i})}(f^{-1}(p), \xi)$ respectively denote the number of fibers in Corollary~\ref{Hordecomp} (ii), (iii) and (iv) that appear after the splitting $\xi$.
Then we have 
\begin{align*}
H_{D}(f^{-1}(p))&=\sum_{i}\frac{m_{i}}{b}D_{i}(f^{-1}(p), \xi)+\sum_{i=0}^{\llcorner \frac{g}{2} \lrcorner}\frac{b_i-b}{b}\delta_{i}(f^{-1}(p), \xi) \\
&+\sum_{i=1}^{k} \left(\frac{n_{i}}{mb}+\frac{s_{D}-4}{2}\left(1-\frac{1}{m} \right)(g-1)\right)\delta_{(m, n_{i})}(f^{-1}(p), \xi).
\end{align*}
Note that the numbers $D_{i}(f^{-1}(p), \xi)$, $\delta_{i}(f^{-1}(p), \xi)$ and $\delta_{(m, n_{j})}(f^{-1}(p), \xi)$ depend on the choices of splitting deformations $\xi$ .

\begin{question}
Can the equation~\eqref{Hordivdecomp} be considered as an equation of Weil divisors on a suitable moduli stack of curves?
\end{question}

Applying Corollary~\ref{Hordecomp} to the cases $g\le 5$, for which the Morsification conjecture is known in the complex analytic setting, one obtains a refinement of the previously known slope equalities. 
In particular, this includes a positive answer of Reid's conjecture {\cite[Conjecture~3.2]{Rei}}:
 
\begin{example} \label{genus2345}
(1) Suppose that $g=2$ and $D=\delta_{1}$.
Then, the Horikawa divisor $H_{\delta_1}$ on $\mathcal{M}^{\star}_{2}=\mathcal{M}^{\mathrm{Gor}}_{2}\setminus \overline{\mathcal{M}}_{2}$ as in Remark~\ref{relMorsMod}~(1) is described by
$$
H_{\delta_1}=\kappa-2\lambda=10\lambda-\delta=\delta_{1}.
$$
Thus, we obtain the slope equality
$$
K_{f}^2=2\chi_{f}+\sum_{p\in B}H_{D}(f^{-1}(p))=2\chi_{f}+\delta_{1}
$$
for any relatively minimal fibered surface $f\colon S\to B$ of genus $2$,
which gives a moduli-theoretic another proof of Horikawa's one ({\cite[Theorem~3]{Hor2}}).
Note that the Horikawa index $H_{D}(f^{-1}(p))=\delta_{1}(f^{-1}(p))$ can be regarded as the virtual number of type $1$ nodes on $f^{-1}(p)$ (Example~\ref{localinvex}~(4)).

\smallskip

\noindent
(2) Suppose that $g=3$ and $D$ is the locus of hyperelliptic curves of genus $3$.
Then, it is known that $D=9\lambda-\delta_{0}-3\delta_{1}$ in the Picard group of $\overline{\mathcal{M}}_{3}$ (\cite{HaMu}).
It follows from \eqref{Hordivdecomp} and Proposition~\ref{smmultsplit}~(3)  that
$$
H_{D}=D+2\delta_{1} +3\delta_{(2,1)},
$$
where we note that the locus of smooth curves of genus $3$ with a free involution is contained in the hyperelliptic locus $D$. 
Hence, we obtain the slope equality
$$
K_{f}^2=3\chi_{f}+\sum_{p\in B}H_{D}(f^{-1}(p))=3\chi_{f}+D+2\delta_{1}+3\delta_{(2,1)}
$$
for any relatively minimal non-hyperelliptic fibered surface $f\colon S\to B$ of genus $3$,
which solves the conjecture of Reid {\cite[Conjecture~3.2]{Rei}} and Ashikaga--Konno {\cite[p.33, Conjecture]{AsKo}}.

\smallskip

\noindent
(3) Suppose that $g=4$ and $D$ is the closure of the locus of genus $4$ curves with unique trigonal pencil.
Then, it is known that $D=34\lambda-4\delta_{0}-14\delta_{1}-18\delta_{2}$ in the Picard group of $\overline{\mathcal{M}}_{4}$ ({\cite[Theorem~2]{EiHa}}).
It follows from \eqref{Hordivdecomp} and Proposition~\ref{smmultsplit}~(1) that
$$
H_{D}=\frac{1}{4}D+\frac{5}{2}\delta_{1} +\frac{7}{2}\delta_{2}+\frac{9}{2}\delta_{(3,0)}.
$$
Note that the locus of smooth curves of genus $4$ with a free automorphism of order $3$ is not contained in $D$.
Indeed, the curve $C$ of genus $4$ defined by the two equations 
$$
x_{0}^{2}+x_{1}^{2}+x_{2}x_{3}=0,\quad x_{0}^{3}+x_{1}^{3}+x_{2}^{3}+x_{3}^{3}=0
$$
 in $\mathbb{P}^3$ does not belong to $D$ and has a free automorphism 
$$
x_{0}\mapsto x_{0},\quad x_{1}\mapsto x_{1},\quad x_{2}\mapsto \omega x_{2},\quad  x_{3}\mapsto \omega^{2}  x_{3}
$$ 
of order $3$, where $\omega=\exp(2\pi\sqrt{-1}/3)$.
Hence, we obtain the slope equality
$$
K_{f}^2=\frac{7}{2}\chi_{f}+\sum_{p\in B}H_{D}(f^{-1}(p))=\frac{7}{2}\chi_{f}+\frac{1}{4}D+\frac{5}{2}\delta_{1} +\frac{7}{2}\delta_{2}+\frac{9}{2}\delta_{(3,0)}
$$
for any relatively minimal fibered surface $f\colon S\to B$ of genus $4$ whose general fiber has exactly two trigonal pencils.

\smallskip

\noindent
(4) Suppose that $g=5$ and $D$ is the closure of the locus of trigonal curves of genus $5$.
Then, it is known that $D=8\lambda-\delta_{0}-4\delta_{1}-6\delta_{2}$ in the Picard group of $\overline{\mathcal{M}}_{5}$ (\cite{HaMu}).
It follows from \eqref{Hordivdecomp} and Proposition~\ref{smmultsplit}~(1) that
$$
H_{D}=D+3\delta_{1} +5\delta_{2}+4\delta_{(2,0)}+6\delta_{(4,0)}.
$$
Note that the locus of smooth curves of genus $5$ with a free automorphism of order $4$ (and hence of order $2$) is not contained in $D$.
Indeed, the curve $C$ of genus $5$ defined by three equations in $\mathbb{P}^{4}$ which are general linear combinations of 
\begin{align*}
&x_{0}^2,\quad  (x_{1}+x_{2}+x_{3}+x_{4})^{2},\quad (x_{1}+\sqrt{-1}x_{2}-x_{3}-\sqrt{-1}x_{4})^{2},  \\
&(x_{1}-\sqrt{-1}x_{2}-x_{3}+\sqrt{-1}x_{4})^{2},\quad x_{1}x_{3}+x_{2}x_{4},\quad  x_{1}x_{2}+x_{2}x_{3}+x_{3}x_{4}+x_{4}x_{1} 
\end{align*}
 is tetragonal and has a free automorphism of order $4$ defined as $x_{0}\mapsto x_{0}$, $x_{i}\mapsto x_{i+1}$ for $i=1,\ldots,4$ where $x_{5}:=x_{1}$.
Hence, we obtain the slope equality
$$
K_{f}^2=4\chi_{f}+\sum_{p\in B}H_{D}(f^{-1}(p))=4\chi_{f}+D+3\delta_{1} +5\delta_{2}+4\delta_{(2,0)}+6\delta_{(4,0)}
$$
for any relatively minimal tetragonal fibered surface $f\colon S\to B$ of genus $5$. 
\end{example}

\subsection{Slope inequality of Chern invariants}
Let $f^{-1}(p)$ be a relatively minimal fiber germ.
A question of Lu and Tan \cite[p.3395, Questions~(3)]{LuTa} asks whether $c_{1}^2(f^{-1}(p))\ge \chi_{f^{-1}(p)}$ holds.
If $f^{-1}(p)$ is isotrivial, it is easily seen by using base change and the usual slope inequality \eqref{usualslopeineq} that a more strong inequality 
$$
c_{1}^2(f^{-1}(p))\ge \frac{4(g-1)}{g}\chi_{f^{-1}(p)}
$$ 
holds (see the proof of Theorem~\ref{slope4}).
In this subsection, we give an affirmative answer of this question if the Morsification conjecture or the Moduli conjecture holds true:

\begin{theorem} \label{LuTanQ}
Assume that $f^{-1}(p)$ satisfies the Morsification conjecture or the Moduli conjecture holds true.
Then, we have $c_{1}^2(f^{-1}(p))\ge \chi_{f^{-1}(p)}$.
\end{theorem}

\begin{proof}
We will apply Theorem~\ref{slopeeq} and Proposition~\ref{Horindcompare} to sufficiently ample divisors $D$.
Note that a divisor $D$ of the form $D=a\lambda-b\delta$ is ample if and only if $b>0$ and $s_{D}>11$ ({\cite[Theorem~(1.3)]{CoHa}}).
For $\varepsilon>0$, we can choose an effective ample divisor $D_{\varepsilon}$ on $\overline{\mathcal{M}}_{g}$ of the form $D_{\varepsilon}=a_{\varepsilon}\lambda-b_{\varepsilon}\delta$ with 
$$
s_{D_{\varepsilon}}=\frac{a_{\varepsilon}}{b_{\varepsilon}}=11+\varepsilon
$$
such that the moduli map of the stable reduction of $f^{-1}(p)$ is disjoint from $D_{\varepsilon}$.
By Theorem~\ref{slopeeq}, the Horikawa index $H_{D_{\varepsilon}}(f^{-1}(p))$ is non-negative.
Thus, applying Proposition~\ref{Horindcompare}, we have
$c_1^{2}(f^{-1}(p))\ge (1-\varepsilon)\chi_{f^{-1}(p)}$ for any $\varepsilon>0$.
Taking the limit as $\varepsilon$ approaches $0$, we obtain $c_1^{2}(f^{-1}(p))\ge \chi_{f^{-1}(p)}$.
\end{proof}

\begin{remark} 
Recently, Cheng and Lu \cite{ChLu} gave an affirmative solution to the Lu--Tan question. 
Their method is completely different from ours.
\end{remark}

\subsection{Examples}
Theorem~\ref{slopeeq} implies that an effective divisor $D$ on $\mathcal{M}_{g}$ with smaller slope $s_{D}$ produces a stronger slope inequality of general fibered surfaces of genus $g$.

\smallskip

\noindent
(1) {\em Brill--Noether divisor} $BN$:
Assume that $g+1$ is not a prime number and write $g+1=(r+1)(s-1)$ for some integers $r\ge 1$ and $s\ge 3$.
Put $d:=rs-1$.
Let $BN_{g,d}^{r}$ denote the locus in $\mathcal{M}_{g}$ parametrizing curves carrying a linear  system $\mathfrak{g}^{r}_{d}$ of degree $d$ and rank $r$.
We also denote by the same symbol $BN_{g,d}^{r}$ its closure in $\overline{\mathcal{M}}_{g}$.
Then it is known that $BN_{g,d}^{r}$ is of codimension $1$ and 
$$
BN_{g,d}^{r}=c\left((g+3)\lambda-\frac{g+1}{6}\delta_0-\sum_{i=1}^{\lfloor g/2 \rfloor}i(g-i)\delta_i \right)
$$
 holds in $\Pic_{\Q}(\overline{\mathcal{M}}_{g})$, 
 where $c$ is a positive rational number depending only on $g$, $r$ and $d$ ({\cite[Theorem~1]{EiHa}}).
In particular, we have 
$$
s(BN_{g,d}^{r})=\frac{6(g+3)}{(g+1)}=6+\frac{12}{g+1}.
$$
Applying Theorem~\ref{slopeeq} to $BN_{g,d}^{r}$, we have the following:

\begin{corollary} \label{BNslopeineq}    
Let $f\colon S\to B$ be a relatively minimal fibered surface of genus $g$ whose general fiber has no $\mathfrak{g}^{r}_{d}$ and assume that any fiber of $f$ satisfies the Morsification conjecture or the Moduli conjecture holds true.
Then we have
$$
K_{f}^{2}= \frac{6(g-1)}{g+1}\chi_{f}+\sum_{p\in B}H_{BN_{g,d}^{r}}(f^{-1}(p)) \ge \frac{6(g-1)}{g+1}\chi_{f}.
$$
\end{corollary}

In particular, the claim with $r=1$ gives a new proof of {\cite[Theorem~4.1]{Kon}} under the assumption of the Morsification conjecture or the Moduli conjecture.
The cases of $g=3,5$ are nothing but Example~\ref{genus2345}~(2), (3).

\smallskip

\noindent
(2) {\em Gieseker--Petri divisor} $GP$: 
Assume that $g$ is not a prime number and write $g=(r+1)(s-1)$ for some integers $r\ge 1$ and $s\ge 3$.
Put $d:=rs$.
Let $GP_{g,d}^{r}$ denote the locus in $\mathcal{M}_{g}$ parametrizing curves carrying a  $\mathfrak{g}^{r}_{d}$ violating the Petri condition,
where a line bundle $L$ on a curve $C$ satisfies the {\em Petri condition} if the multiplication map
$$
H^{0}(L)\otimes H^{0}(\omega_{C}\otimes L^{-1})\to H^{0}(\omega_{C})
$$
is injective.
We also denote by the same symbol $GP_{g,d}^{r}$ its closure in $\overline{\mathcal{M}}_{g}$.
If $r=1$, then $d=s=g/2+1$ and it is known that $GP_{g,d}^{1}$ is a divisor satisfying
$$
GP_{g,d}^{1}=c\left(e\lambda-f_{0}\delta_0-\sum_{i=1}^{\lfloor g/2 \rfloor}f_{i}\delta_i \right)
$$
 in $\Pic_{\Q}(\overline{\mathcal{M}}_{g})$, 
where 
$$
c=2\frac{(2d-4)!}{d!(d-2)!},\quad e=6d^2+d-6,\quad f_0=d(d-1),\quad f_1=(2d-3)(3d-2), 
$$
and other coefficients $f_{i}$, $i\ge 2$ are explicitly given and greater than $f_{0}$ ({\cite[Theorem~2]{EiHa}}).
In particular, we have 
$$
s(GP_{g,d}^{1})=\frac{6d^2+d-6}{d(d-1)}=6+\frac{2(7g+2)}{g(g+2)}.
$$
Similar results about $GP_{g,d}^{r}$ for general $r$ are known ({\cite[Theorem~1.6]{Far}}).
Applying Theorem~\ref{slopeeq} to $GP_{g,d}^{1}$, we have the following:

\begin{corollary} \label{GPslopeineq}
Let $f\colon S\to B$ be a relatively minimal fibered surface of even genus $g$ such that  the gonality of general fibers equals $d=g/2+1$ and any gonality pencil satisfies the Petri condition.
Assume that any fiber of $f$ satisfies the Morsification conjecture or the Moduli conjecture holds true.
Then we have
$$
K_{f}^{2}=\frac{2(3g+2)(g-1)}{g(g+2)}\chi_{f}+\sum_{p\in B}H_{GP_{g,d}^{1}}(f^{-1}(p)) \ge \frac{2(3g+2)(g-1)}{g(g+2)}\chi_{f}.
$$
\end{corollary}
The case for $g=4$ is nothing but Example~\ref{genus2345}~(3), the inequality of which was first proved in \cite{Che}, {\cite[Theorem~4.1]{Kon3}}.

Recently, Okuda \cite{Oku} proved the Morsification conjecture for $g=6$.
Thus, we have the following slope equality for general fibered surfaces of genus $6$:
\begin{corollary} \label{g6slopeineq}
Let $f\colon S\to B$ be a relatively minimal fibered surface of genus $6$ such that  general fibers are tetragonal and any gonality pencil satisfies the Petri condition.
Then, we have
$$
K_{f}^{2}=\frac{25}{6}\chi_{f}+\sum_{p\in B}H_{GP_{6,4}^{1}}(f^{-1}(p)) \ge \frac{25}{6}\chi_{f}.
$$
\end{corollary}
This is already a new result even when restricted to the slope inequality.

\smallskip

\noindent
(3) {\em K3 divisor} $K$: 
Let $\mathcal{F}_{g}$ denote the moduli stack of polarized K3 surfaces of genus $g$.
Let $\mathcal{P}_{g}$ denote the universal $\mathbb{P}^{g}$-bundle on $\mathcal{F}_{g}$ defined by the polarizations, that is, the moduli stack of pairs $((S, H), C)$ where $(S, H)$ is a polarized K3 surface of genus $g$ and $C\in |H|$.
Then, the projection $((S, H), C)\mapsto C$ defines a rational map $\mathcal{P}_{g}\dasharrow \mathcal{M}_{g}$.
It is known that this map is dominant if and only if $1\le g\le 11$ and $g\neq 10$ ({\cite[Theorems~0.7, 6.1]{Muk}}, {\cite[Corollary~2]{MoMu}}).
Now we consider $g=10$ and denote by $K$ the closure in $\overline{\mathcal{M}}_{10}$ of the image of $\mathcal{P}_{10}\dasharrow \mathcal{M}_{10}$.
It is an effective divisor on $\overline{\mathcal{M}}_{10}$ and satisfies
$$
K=7\lambda-\delta_0-5\delta_1-9\delta_2-12\delta_3-14\delta_4-15\delta_5
$$
in $\Pic_{\Q}(\overline{\mathcal{M}}_{10})$ ({\cite[Theorem~1.6]{FaPo}}, {\cite[Theorem~2.18]{Far2}}).
In particular, we have $s(K)=7$.

Applying Theorem~\ref{slopeeq} to $K$, we have the following:
\begin{corollary}
Let $f\colon S\to B$ be a relatively minimal fibered surface of genus $10$ whose general fiber does not lie on a limit of K3 surfaces. 
Assume that any fiber of $f$ satisfies the Morsification conjecture or the Moduli conjecture holds true.
Then we have
$$
K_{f}^{2}=5\chi_{f}+\sum_{p\in B}H_{K}(f^{-1}(p)) \ge 5\chi_{f}.
$$
\end{corollary}

\subsection{Relation to the Slope conjecture}
Harris and Morrison \cite{HaMo} predicted that $s_{D}$ is greater than or equal to $6+12/(g+1)$ for any non-zero effective divisor $D$ on $\overline{\mathcal{M}}_{g}$ containing no boundary divisors with equality holding if and only if $D$ is a linear combination of Brill--Noether divisors, which is called the {\em Slope conjecture}.
This conjecture is false in general and the first counterexample to be discovered was the K3 divisor $K$ on $\overline{\mathcal{M}}_{10}$ (\cite{FaPo}).
For other counterexamples and the current status for the Slope conjecture, see for example \cite{Far}, \cite{Far2} and \cite{FJP}.
Let $s_{g}$ denote the infimum of $s_{D}$ among all non-zero effective divisors $D$ on $\overline{\mathcal{M}}_{g}$ containing no boundary divisors.
For $g\le 11$, the numbers $s_{g}$ are known as in Table~\ref{sglist} (\cite{HaMo}, \cite{Tan}, \cite{FaPo}).
\begin{table}[h]

\caption{$s_g$ for $g\le 11$}
  \label{sglist}
\begin{tabular}{c|cccccccccc}
$g$ & $2$ & $3$ & $4$ & $5$ & $6$ & $7$ & $8$ & $9$ & $10$ & $11$ \\ \hline
$s_g$ & $10$ & $9$ & $\frac{17}{2}$ & $8$ & $\frac{47}{6}$ & $\frac{15}{2}$ & $\frac{22}{3}$ & $\frac{36}{5}$ & $7$ & $7$ \\
$D$ with $s_g=s_{D}$ & $\delta_1$ & $BN$ & $GP$ & $BN$ & $GP$ & $BN$ & $BN$ & $BN$ & $K$ & $BN$
\end{tabular}

\end{table}
In particular, the Slope conjecture is true for $g\le 11$ and $g\neq 10$.
Theorem~\ref{slopeeq} implies that
if $s_{g}$ attains some effective divisor $D$ on $\overline{\mathcal{M}}_{g}$, 
the slope of relatively minimal general fibered surfaces of genus $g$ is greater than or equal to $12-s_{g}$ 
under the assumption of the Morsification conjecture or the Moduli conjecture.

In \cite{CFM}, Chen conjectured that $\lim_{g\to \infty}s_{g}=0$ while Farkas and Morrison conjectured that $\lim_{g\to \infty}s_{g}=6$.
The following applications of Theorem~\ref{slopeeq} give little evidence for the later conjecture:

\begin{theorem} \label{slope4}
Assume that the Morsification conjecture or the Moduli conjecture holds true.
Let $D$ be an effective divisor on $\overline{\mathcal{M}}_{g}$ which contains no boundary divisors.
If $s_{D}\le 4$, then $D$ contains the locus of smooth curves with non-trivial automorphisms with fixed points.
\end{theorem}

\begin{proof}
Assume contrary that $D$ does not contain the moduli point of a smooth curve $C$ of genus $g$ with an automorphism $\sigma$ of order $m\ge 2$ having fixed points.
If $s_{D}<4$, we can  replace $D$ by $D+A$ for an ample effective $\Q$-divisor $A$ which does not contain $C$ with $s_{D+A}=4$.
Thus we may assume that $s_{D}=4$.
We construct an isotrivial fibered surface $f\colon S\to B$ with $m$ singular fibers isomorphic each other and with the general fiber isomorphic to $C$ as follows:
First we consider the standard cyclic action of $\sigma$ to $\mathbb{P}^{1}$ of order $m$ and the quotient of the diagonal action $S_{1}:=(C\times \mathbb{P}^1)/\sigma$.
Next we take a cyclic covering $B\to \mathbb{P}^{1}\cong \mathbb{P}^{1}/\sigma$ of order $m-1$ totally branched at $\infty$ and unramified over $0$.
Taking the base change, we obtain a fibration $S_{1}\times_{\mathbb{P}^{1}}B\to B$.
Taking a resolution and the relatively minimal model, we obtain a desired fibration $f\colon S\to B$.
Note that each singular fiber $f^{-1}(p)$ is completely determined by the action $\sigma$ to $C$, and has at least two irreducible components since $\sigma$ has fixed points and hence $S_{1}$ is singular.
Moreover, $f^{-1}(p)$ is not a multiple of a nodal curve since it is isotrivial.
Then we have $c_{1}^{2}(f^{-1}(p))<8\chi_{f^{-1}(p)}$ by the Miyaoka--Yau type inequality.
It follows from the local-to-global formula that
$$
K_{f}^{2}=mc_{1}^{2}(f^{-1}(p))<8m\chi_{f^{-1}(p)}=8\chi_{f}.
$$
On the other hand, Theorem~\ref{slopeeq} implies that $K_{f}^{2}\ge (12-s_{D})\chi_{f}=8\chi_{f}$, which is a contradiction.
\end{proof}

\begin{definition} \label{reducedtypedef}
For a smooth curve $C$ of genus $g\ge 2$, a non-trivial automorphism $\sigma\in \mathrm{Aut}(C)$ is {\em of reduced type} if the fiber $f^{-1}(p)$ constructed as in the proof of Theorem~\ref{slope4} has the reduced canonical model.
For example, the covering transformation of a simple cyclic covering $C\to \mathbb{P}^1$ of degree $m$ branched at $m$ or $2m$ points and an automorphism of maximal order $4g+2$ are of reduced type.
\end{definition}

\begin{theorem} \label{slope6}
Let $D$ be an effective divisor on $\overline{\mathcal{M}}_{g}$ which contains no boundary divisors.
If $s_{D}< 6$, then $D$ contains the locus of smooth curves with automorphisms of reduced type.
\end{theorem}

\begin{proof}
The proof is similar to the proof of Theorem~\ref{slope4} by replacing the Miyaoka--Yau type inequality $c^{2}_{1}(f^{-1}(p))<8\chi_{f^{-1}(p)}$ with the stronger version $c^{2}_{1}(f^{-1}(p))\le 6\chi_{f^{-1}(p)}$.
Note that Theorem~\ref{slopeeq} can be applied by Proposition~\ref{splitcri}.
\end{proof}

\subsection{Relation to the geography of surfaces of general type}
It is classically known that slope inequalities of the form $K_{f}^{2}\ge (12-s_{g})\chi_{f}$ are related to the geography of surfaces of general type.

In the case $g=2$, the slope inequality $K_{f}^{2}\ge 2\chi_{f}$ holds for genus $2$ fibrations $f\colon S\to B$ (\cite{Hor2}), which is now a special case of the usual slope inequality \eqref{usualslopeineq} (\cite{Xia}, \cite{CoHa}).
This is equivalent to the Noether inequality $K_{S}^{2}\ge 2\chi(\O_{S})-6$ when $B=\mathbb{P}^1$.
Almost all surfaces $S$ of general type near the Noether line naturally admits genus $2$ fibrations over $\mathbb{P}^1$ (\cite{Hor}).

For $g=3$, then the slope inequality $K_{f}^{2}\ge 3\chi_{f}$ holds for non-hyperelliptic fibered surfaces $f\colon S\to B$ of genus $3$ (\cite{Rei}, \cite{Kon3}).
This is equivalent to the Castelnuovo inequality $K_{S}^{2}\ge 3\chi(\O_{S})-10$ when $B=\mathbb{P}^1$.
Almost all surfaces $S$ whose canonical maps are birational onto image near the Castelnuovo line naturally admit non-hyperelliptic genus $3$ fibrations over $\mathbb{P}^1$ (\cite{AsKo}, \cite{Kon2}).

For $g=5$, the slope inequality $K_{f}^{2}\ge 4\chi_{f}$ holds for tetragonal fibered surfaces $f\colon S\to B$ of genus $5$ (\cite{Kon3}, \cite{Kon}).
This is equivalent to Reid's conjectual inequality $K_{S}^{2}\ge 4\chi(\O_{S})-16$ (assuming the irregularity of $S$ vanishes) when $B=\mathbb{P}^1$.
Reid's conjecture, which is a surface analog of the Enriques--Babbage--Petri theorem for curves, states the following:

\begin{conjecture}[Reid's conjecture {\cite[p.541]{Rei2}}]  \label{Reidconj}
Let $S$ be a minimal surface of general type whose canonical map is birational onto the image $X\subseteq \mathbb{P}^{p_{g}(S)-1}$.
If $X$ is the intersection of all hyperquadrics through $X$, then $K_{S}^{2}\ge 4p_{g}(S)-12$ holds.
\end{conjecture}
This conjecture is still open although it is true for even surfaces with irregularity $0$ ({\cite[Theorem~8.4]{Kon5}}).
According to \cite{Kon4}, all even surfaces $S$ with irregularity $0$ satisfying the assumption of Conjecture~\ref{Reidconj} with $K_{S}^{2}=4p_{g}(S)-12$ are classified into 5 classes.
One of these classes consists of surfaces $S$ that admit tetragonal genus $5$ fibrations $f\colon S\to \mathbb{P}^1$ with $K_{f}^{2}=4\chi_{f}$.

For $g=10$ (resp.\ $g=11$), the slope inequality $K_{f}^{2}\ge 5\chi_{f}$ is equivalent to $K_{S}^{2}\ge 5\chi(\O_{S})-27$ (resp.\ $K_{S}^{2}\ge 5\chi(\O_{S})-30$).
Here we propose the following natural question:

\begin{question}
What geometric condition on surfaces give rise to the inequality $K_{S}^{2}\ge 5\chi(\O_{S})-27$ $($resp.\  $K_{S}^{2}\ge 5\chi(\O_{S})-30$$)$?
Do general fibered surfaces of genus $10$ $($resp.\ genus $11$$)$ over $\mathbb{P}^1$ form a major class of surfaces with this geometric condition near the line $K_{S}^{2}=5\chi(\O_{S})-27$ $($resp.\  $K_{S}^{2}=5\chi(\O_{S})-30$$)$?
\end{question}

\end{document}